\documentclass{siamltex}
\usepackage{graphicx}
\usepackage{bm}
\usepackage{amsmath}
\usepackage{amssymb}
\usepackage{epsfig}
\usepackage{subfigure}
\usepackage{latexsym}
\usepackage{amsbsy}
\usepackage{lineno}
\usepackage{xparse}
\usepackage{comment}
\usepackage[usenames, dvipsnames]{color}
\usepackage[breaklinks,colorlinks=true,allcolors=black]{hyperref}

\NewDocumentCommand{\dgal}{sO{}m}{%
  \IfBooleanTF{#1}
    {\dgalext{#3}}
    {\dgalx[#2]{#3}}%
}
\NewDocumentCommand{\dgalext}{m}{%
  \sbox0{%
    \mathsurround=0pt 
    $\left\{\vphantom{#1}\right.\kern-\nulldelimiterspace$%
  }%
  \sbox2{\{}%
  \ifdim\ht0=\ht2
    \{\kern-.45\wd2 \{#1\}\kern-.45\wd2 \}%
  \else
    \left\{\kern-.5\wd0\left\{#1\right\}\kern-.5\wd0\right\}%
  \fi
}

\NewDocumentCommand{\dgalx}{om}{%
  \sbox0{\mathsurround=0pt$#1\{$}%
  \sbox2{\{}%
  \ifdim\ht0=\ht2
    \{\kern-.45\wd2 \{#2\}\kern-.45\wd2 \}%
  \else
    \mathopen{#1\{\kern-.5\wd0 #1\{}
    #2
    \mathclose{#1\}\kern-.5\wd0 #1\}}
  \fi
}

\newtheorem{assumption}{Assumption}[section]
\newtheorem{algorithm}{Algorithm}[section]

\newcommand{\be}{\begin{equation}}
\newcommand{\ee}{\end{equation}}
\newcommand{\bse}{\begin{subequations}}
\newcommand{\ese}{\end{subequations}}
\newcommand{\bea}{\begin{eqnarray}}
\newcommand{\eea}{\end{eqnarray}}
\newcommand{\beas}{\begin{eqnarray*}}
\newcommand{\eeas}{\end{eqnarray*}}
\newcommand{\ba}{\begin{array}}
\newcommand{\ea}{\end{array}}

\newcommand{\sigmab}{\mbox{\boldmath$\sigma$}}

\font\msbm=msbm10

\newcommand{\R}{\hbox{{\msbm \char "52}}}

\newcommand{\euh}[2][]{#1\bm{e}_{h,#2}^{\bm{u}}}
\newcommand{\etph}[2][]{#1e_{h,#2}^{\tilde p}}
\newcommand{\balpha}{\bm{\alpha}}

\newcommand{\GI}{\mathfrak I}

\newcommand{\GQ}{\mathfrak Q}

\newcommand{\reff}[2]{\stackrel{\eqref{#1}}{#2}}	

\definecolor{otherblue}{rgb}{0,0.3,0.6}
\def\rblf#1{\textcolor{black}{#1}}

\def\todo#1{\textcolor{red}{#1}}
\def\rbl#1{{\textcolor{black}{#1}}}  

\def\rrb#1{{\textcolor{black}{#1}}}

\title{{Robust a posteriori error estimation for stochastic Galerkin formulations of parameter-dependent linear elasticity equations}\thanks{{This work was supported 
by  EPSRC grants EP/P013317/1 and EP/P013791/1. The authors would also like to thank the Isaac Newton Institute for Mathematical Sciences, Cambridge, for support and hospitality during the Uncertainty Quantification programme which was supported by EPSRC {grant} EP/K032208/1. The third author acknowledges the support of the Simons Foundation.}}}

\author{
Arbaz Khan\thanks{
School of Mathematics, University of Manchester, UK (\tt{arbaz.khan@manchester.ac.uk})}
\and
Alex Bespalov\thanks{
School of Mathematics, University of Birmingham, UK (\tt{A.Bespalov@bham.ac.uk})}
\and
Catherine E. Powell\thanks{
School of Mathematics, University of Manchester, UK (\tt{c.powell@manchester.ac.uk})}
\and
David J. Silvester\thanks{
School of Mathematics, University of Manchester, UK (\tt{d.silvester@manchester.ac.uk})}.
}

\begin{document}

\maketitle

\begin{abstract} The focus of this work is a posteriori error estimation for stochastic Galerkin approximations of parameter-dependent linear elasticity equations. The starting point is a three-field PDE model in which the Young's modulus is an affine function of a countable set of parameters. We analyse the weak formulation, its stability with respect to a weighted norm and \rbl{discuss} approximation using stochastic Galerkin mixed finite element methods (SG-MFEMs).  We introduce a novel a \rbl{posteriori} error estimation scheme and establish upper and lower bounds for the SG-MFEM error. \rbl{The} constants in the bounds are independent of the Poisson ratio as well as the SG-MFEM discretisation parameters. In addition, we \rbl{discuss} proxies for the error reduction associated with certain enrichments of the SG-MFEM spaces and \rbl{we} use these to develop an adaptive algorithm that terminates when the estimated error falls below a user-prescribed tolerance. We prove that both the a posteriori error estimate and the error reduction proxies are reliable and efficient in the incompressible limit case.  Numerical results are presented to validate the theory. \rbl{All experiments were performed using open source (IFISS) software that is available online.} 
\end{abstract}



\pagestyle{myheadings}

\thispagestyle{plain}
\markboth{A.\ KHAN, A. \ BESPALOV, C.\  E.\ POWELL and D.\  J.\  SILVESTER\ \ $|$\ \ \today}
{A posteriori error estimation for linear elasticity problems with {uncertain inputs}}

\section{Introduction}\label{sec11} 

The motivation for this work is the need to develop accurate and efficient numerical algorithms for solving linear elasticity problems in engineering applications where the Young's modulus $E$ of the material considered is spatially varying in an uncertain way.  Of particular interest is the nearly incompressible case, which poses a significant challenge for numerical methods, even when all the model inputs are known exactly.  A well known strategy for avoiding locking of finite element methods for standard elasticity problems is to introduce an auxiliary pressure variable, obtain a coupled system of partial differential equations (PDEs) and then apply mixed finite element methods, \cite{RLH, DR, KPS}.   In \cite{KPSpre}, Khan et al.~introduced a three-field PDE model with parameter-dependent Young's modulus which is amenable to discretisation by stochastic Galerkin mixed finite element methods (SG-MFEMs).  We revisit this model and introduce a novel a posteriori error estimation strategy for SG-MFEM approximations that is robust in the incompressible limit case. 

There  is little work to date on a posteriori error estimation and adaptivity for stochastic Galerkin approximations of \emph{mixed} formulations of linear elasticity problems. In \cite{MBBGS}, Matthies et al.~give a comprehensive overview of how to incorporate uncertainty into material parameters in linear elasticity problems and discuss stochastic finite element methods.  Several other works in the engineering literature also cover numerical aspects of implementing stochastic finite element methods for elasticity problems. For example, see \cite{ghanem_spanos, shang2013} and references therein.  A framework for residual-based a posteriori error estimation and adaptive stochastic Galerkin approximation for second order linear elliptic PDEs is presented in \cite{EIGEL2014247}. Numerical results are presented for planar linear elasticity problems but mixed formulations are not considered.  A priori analysis for so-called best $N$-term approximations of standard mixed formulations of stochastic and multiscale elasticity problems is provided in \cite{Hoang2014, Hoang2016}. 

The formal mathematical specification of the problem we consider is as follows. Let $D$ (the spatial domain) be a bounded Lipschitz polygon in $\R^2$ (polyhedron in $\R^3$) with boundary $\partial D = \partial D_D \cup \partial D_N$, where $\partial D_D \cap \partial D_N= \emptyset$ and $\partial D_D, \partial D_N\neq \emptyset$. Next, we introduce a vector of countably many parameters $\bm{y}=(y_1,y_2,\ldots)$ with each $y_{k} \in \Gamma_{k}=[-1,1].$ We model the Young's modulus in the linear elasticity equations as a parameter-dependent function of the form
\begin{align}\label{ecoef11}
E(\bm{x},\bm{y}):= e_0(\bm{x})+\sum_{k=1}^{\infty}e_k(\bm{x}) y_{k}, \quad \bm{x}\in D,\ \ \bm{y} \in \Gamma.
\end{align}
In \eqref{ecoef11}, $\Gamma = \Pi_{k=1}^{\infty} \Gamma_{k}$ denotes the parameter domain and $e_{0}$ typically represents the mean of $E$. The parameters $y_{k}$ are images of mean zero random variables and these encode our uncertainty about $E$. Picking a specific $\mathbf{y} \in \Gamma$ corresponds to generating a realisation of $E$. \rbl{As in \cite{KPSpre}, we consider the parametric problem}:
find $\bm{u}: D\times\Gamma \to \rrb{\mathbb{R}^{d}}$ and ${{p}}, \tilde{p}:D\times\Gamma \to \mathbb{R}$ such that
\begin{subequations}\label{os3}
\begin{align}\label{3fieldfor11}
 -\nabla\cdot\sigmab(\bm{x},\bm{y})& =\bm{f}(\bm{x}), \quad \bm{x}\in D, \;\bm{y}\in\Gamma, \\
 \nabla\cdot\bm{u}(\bm{x},\bm{y})+ \tilde{\lambda}^{-1} \tilde{p}(\bm{x},\bm{y}) &=0,\quad \quad \, \, \, \bm{x}\in D, \;\bm{y}\in\Gamma, \label{3fieldfor12}\\
\tilde{\lambda}^{-1} {p}(\bm{x},\bm{y}) - \tilde{\lambda}^{-1} E(\bm{x},\bm{y}) \tilde{p}(\bm{x},\bm{y}) &=0,\quad \quad \, \, \, \bm{x}\in D, \;\bm{y}\in\Gamma, \label{3fieldfor13}\\
 \bm{u}(\bm{x},\bm{y})&=\bm{g}(\bm{x}),\quad \bm{x} \in \partial D_D,\;\bm{y}\in\Gamma,\\
 \sigmab(\bm{x},\bm{y})\, \bm{n}&= \mathbf{0},\quad \quad \, \, \, \bm{x} \in \partial D_N,\;\bm{y}\in\Gamma.
\end{align}      
\end{subequations}
Here, $\sigmab : D\times\Gamma\rightarrow \R^{d\times d}\, (d=2,3)$ is the stress tensor,
$\bm{f}: D\rightarrow \R^{d}$ is the body force, $\bm{n}$ denotes the outward unit normal vector to $\partial D_N$,
$\bm{u}$ is the displacement (the solution field of main interest) and the auxiliary variables that we have introduced are $p: = -\lambda \nabla \cdot \bm{u}$ (the so-called Herrmann pressure \cite{RLH}) and $\tilde{p}:=p/E$.  Recall that $\sigmab$ is related to the strain tensor  $\bm{\varepsilon} : D\times\Gamma\rightarrow \R^{d\times d}$ through the identities 
${\sigmab}= 2 \mu \bm{\varepsilon}  - p \bm{I}$ and 
$\bm{\varepsilon} = \frac{1}{2} \left(\nabla \bm{u} + (\nabla \bm{u}\right)^{\top})$. 
The Lam\'{e} coefficients are 
\begin{align*}
\mu(\bm{x},\bm{y})
=\frac{E(\bm{x},\bm{y})}{2(1+\nu)}, \quad \lambda(\bm{x},\bm{y})=\frac{E(\bm{x},\bm{y})\nu}{(1+\nu)(1-2\nu)}
\end{align*}
and we have also introduced the constant
\begin{align*}
\tilde{\lambda} := {\frac{\lambda(\bm{x},\bm{y})}{E(\bm{x},\bm{y})}}= {\frac{\nu}{(1+\nu)(1-2\nu)}}.
\end{align*}
Note that we assume that the Poisson ratio $\nu$ (and hence $\tilde{\lambda}$) is a given fixed constant 
and that $0<\mu_1<\mu<\mu_2 <\infty$ and $0<\lambda<\infty$ a.e.~in $D \times \Gamma$. 
In contrast to other mixed formulations of the linear elasticity equations, the 
advantage of (\ref{os3}) is that while $E$ appears in the first and third equations, $E^{-1}$ 
does not appear at all. As a result, since $E$ has the affine form \eqref{ecoef11}, the discrete problem associated with SG-MFEM approximations has a structure that is relatively easy to exploit.

In Section \ref{EPWRD} we introduce the weak formulation of \eqref{os3} and discuss stability and well-posedness.  A key feature of our analysis is that we work with a $\nu$-dependent norm. In Section \ref{disformul}, we discuss SG-MFEM approximation and set up the associated finite-dimensional weak problem.  In Section \ref{sec4} we \rbl{describe} our a posteriori error estimation strategy, which requires the solution of four simple problems on detail spaces and the evaluation of one residual, and establish two-sided bounds for the true error in terms of the proposed {estimate}. The constants in the bounds are independent of the Poisson ratio, so that the {estimate} is robust in the incompressible limit $\nu \to 1/2$. In Section \ref{sec5} we 
 \rbl{introduce proxies  for the error reductions} that would be achieved by performing finite element mesh refinement, or by enriching the parametric part of the approximation space.  We establish two-sided bounds, showing that these \rbl{error reduction proxies} are efficient and reliable {and then use them} to develop an adaptive SG-MFEM algorithm. In Section \ref{sec6}, we briefly discuss the incompressible limit case. Finally, in Section \ref{sec7} we present numerical results.

\section{Weak formulation}\label{EPWRD}
Before stating the weak formulation of \eqref{os3}, we give conditions on the Young's modulus that are required to establish well-posedness and define appropriate solution spaces. Recall that $E$ has the form \eqref{ecoef11} with $y_{k} \in [-1,1]$.

\begin{assumption}\label{Assump2}
{The random field} $E\in L^{\infty}(D\times \Gamma)$ and is uniformly bounded away from zero. That is, there exist positive constants $E_{\min}$ and $E_{\max}$ such that
\begin{align}\label{bounde11}
0<E_{\min}\le E(\bm{x},\bm{y}) \le E_{\max} <\infty \quad \mbox{\rm{a.e. in}}\, D\times \Gamma.
\end{align}
{To ensure the lower bound in \eqref{bounde11} is satisfied}, we further assume that 
\be \label{bounde11x}
      0<e_0^{\min}\le e_0(\bm{x})\le e_0^{\max} < \infty \ \ \mbox{a.e. in}\ D \quad
      \mbox{and} \quad
      \frac{1}{e_0^{\min}}\sum_{k=1}^{\infty}|| e_k ||_{L^{\infty}(D)}<1.
\ee
\end{assumption}

Let $\pi(\bm{y})$ be a product measure with $\pi(\bm{y}):= \Pi_{k=1}^{\infty}\pi_k(y_{k})$, where $\pi_k$ is a measure on $(\Gamma_k, \mathcal{B}(\Gamma_k))$ and $\mathcal{B}(\Gamma_k)$ is the Borel $\sigma$-algebra on $\Gamma_k=[-1,1]$. We will assume that the parameters $y_{k}$ in \eqref{ecoef11} are images of independent uniform random variables $\xi_{k} \sim U(-1,1)$ and choose $\pi_{k}$ to be the associated probability measure. In this case, $\pi_{k}$ has density $\rho_{k} =1/2$ with respect to Lebesgue measure and
\begin{align*}
\int_{\Gamma_{k}} y_{k} \, d \pi_{k}(y_{k}) = \int_{\Gamma_{k}} y_{k} (1/2) d y_{k} = 0. 
\end{align*}
Now, given a normed linear space $X(D)$ of real-valued  functions on $D$ (either vector or scalar valued)
 with norm $||\cdot||_X$, we can define the Bochner space
\begin{align*}
L^2_{\pi}(\Gamma, X(D)) :=\left \{v(\bm{x},\bm{y}) : D\times\Gamma\rightarrow \mathbb{R}; ||v||_{L^2_{\pi}(\Gamma, X(D))} <\infty\right\},
\end{align*}
where
\begin{align}\label{bochner_norm}
||\cdot||_{L^2_\pi(\Gamma,X(D))} := \left(\int_{\Gamma}||\cdot||_X^2 d\pi(\bm{y})\right)^{1/2}.
\end{align} 
In particular, we will need the spaces
\begin{align}
\bm{\mathcal{V}} := L^2_{\pi}(\Gamma, \bm{H}^1_{E_0}(D)),\quad {\mathcal W}  := L^2_{\pi}(\Gamma, {L}^2(D))  \quad \mbox{and} \quad\bm{\mathcal{W}} := L^2_{\pi}(\Gamma, \bm{L}^2(D)),
\end{align}
where $\bm{H}^1_{E_0}(D)= \{ \bm{v} \in\bm{H}^1(D),  \bm{v} |_{\partial D_D}=\bm{0}\}$ 
and $\bm{H}^1(D)=\bm{H}^1(D;\mathbb{R}^{{d}})$ is the usual vector-valued Sobolev space. We denote the norms \eqref{bochner_norm} associated with $\bm{\mathcal{V}}$, ${\mathcal{W}}$ and $\bm{\mathcal{W}}$ by $|| \cdot ||_{\bm{\mathcal{V}}}$, $|| \cdot ||_{\mathcal{W}}$ and $|| \cdot ||_{\bm{\mathcal{W}}}$, respectively. 

Assume that the load function $\bm{f}\in ( L^{2}(D))^{{d}}$ and {the} boundary data $\bm{g}= \bm{0}$ on $\partial D_D$. Then, the weak formulation of 
\eqref{os3} is: find $(\bm{u},p,\tilde{p})\in \bm{\mathcal{V}}\times\mathcal{W}\times\mathcal{W}$ such that
\begin{subequations} \label{scm11a}
\begin{align}
a(\bm{u},\bm{v})+b(\bm{v},p)&=f(\bm{v}) \quad\forall \bm{v}\in \bm{\mathcal{V}},\\
b(\bm{u},q)-c(\tilde{p},q)&=0 \quad \quad \, \, \forall q\in \mathcal{W},\\
-c(p,\tilde{q})+d(\tilde{p},\tilde{q})&=0 \, \, \quad \quad \forall \tilde{q}\in \mathcal{W},
\end{align}
\end{subequations}
where we have
\begin{align}
   a(\bm{u},\bm{v}) &:=
   \alpha \int_{\Gamma}\int_{D}
   E(\bm{x},\bm{y}) \bm{\varepsilon}(\bm{u}(\bm{x},\bm{y})):\bm{\varepsilon}(\bm{v}(\bm{x},\bm{y}))d\bm{x}d\pi(\bm{y}),
   \label{fbf_a}
   \\
   b(\bm{v},p) &:=
   -\int_{\Gamma}\int_{D}  {p}(\bm{x},\bm{y}) \nabla \cdot \bm{v}(\bm{x},\bm{y}) d\bm{x}d\pi(\bm{y}),
   \label{fbf_b}
   \\
   c(p,q) &:=
   ({\alpha}\beta)^{-1}\int_{\Gamma}\int_{D} {p}(\bm{x},\bm{y}){q}(\bm{x},\bm{y})d\bm{x}d\pi(\bm{y}),
   \label{fbf_c}
   \\
   d(p,q) &:=
   ({\alpha}\beta)^{-1}\int_{\Gamma}\int_{D} E(\bm{x},\bm{y}) {p}(\bm{x},\bm{y}){q}(\bm{x},\bm{y})d\bm{x}d\pi(\bm{y}),
   \label{fbf_d}
   \\
   f(\bm{v})& : = \int_{\Gamma}\int_{D} f(\bm{x}) \bm{v}(\bm{x},\bm{y})d\bm{x}d\pi(\bm{y}),
   \label{fbf_f}
\end{align}
and we define the constants
\begin{align}
\alpha := \frac{1}{1+\nu}, \qquad \beta := \frac{\nu}{(1-2\nu)}.
\end{align}
Note that $\alpha \beta = \tilde\lambda$. Furthermore, it is important to note that {the constants} $\alpha$ and $\beta$ depend only on
the Poisson ratio $\nu$ which is assumed to be known. It will also be useful to define the combined bilinear form
\begin{align}\label{big_B}
\mathcal{B}(\bm{u},p,\tilde{p}; \bm{v},q,\tilde{q})=a(\bm{u},\bm{v})+b(\bm{v},p)+b(\bm{u},q)-c(\tilde{p},q)-c(p,\tilde{q})+d(\tilde{p},\tilde{q}),
\end{align}
so as to express (\ref{scm11a}) more compactly as: find $(\bm{u},p,\tilde{p})\in \bm{\mathcal{V}}\times\mathcal{W}\times\mathcal{W}$ such that 
\begin{align}\label{scm12}
\mathcal{B}(\bm{u},p,\tilde{p}; \bm{v},q,\tilde{q})=f(\bm{v}) \quad \forall (\bm{v},q,\tilde{q})\in \bm{\mathcal{V}}\times\mathcal{W}\times\mathcal{W}.
\end{align}
In the limit $\nu \to 1/2$, observe that $\beta \to \infty$ and so $c(\cdot, \cdot)$ and $d(\cdot, \cdot)$ disappear from \eqref{scm11a} and \eqref{scm12}, yielding a standard Stokes-like system. See Section \ref{sec6} for more details.

{Proofs of Lemmas \ref{tecre11}--\ref{welpose11} below can be found in \cite{KPSpre} and references therein}. Since these results are important for our analysis, we restate them here for completeness. The first result states that the bilinear forms in \eqref{scm11a} are bounded. The second states that three of them are coercive and $b(\cdot, \cdot)$ satisfies an inf-sup condition on $\bm{\mathcal{V}} \times {\mathcal W}$.

\begin{lemma}\label{tecre11}
 If $E$ satisfies Assumption \ref{Assump2}, then we have
\begin{align}
a(\bm{u},\bm{v}) & \le {\alpha} {E}_{\max}||\nabla\bm{u}||_{\bm{\mathcal{W}}} 
||\nabla\bm{v}||_{\bm{\mathcal{W}}} \quad \forall \bm{u}, \bm{v} \in \bm{\mathcal{V}}, \\
b(\bm{u}, p)&\le {\sqrt{d}} \, ||\nabla\bm{u}||_{\bm{\mathcal{W}}}||p||_{\mathcal{W}}
\quad \quad \quad \, \, \forall \bm{u} \in \bm{\mathcal{V}}, \, \forall p \in \mathcal{W}, \\
c(p,q)&\le ({\alpha}\beta)^{-1} ||p||_{\mathcal{W}} ||q||_{\mathcal{W}} \quad \quad \quad\forall p,q\in\mathcal{W},\\
d(p,q)&\le ({\alpha}\beta)^{-1} {E}_{\max} ||p||_{\mathcal{W}} ||q||_{\mathcal{W}}\quad\forall p,q\in\mathcal{W}.
\end{align}
\end{lemma}
\begin{lemma}\label{tecre12}
If $E$ satisfies Assumption \ref{Assump2}, then we have
\begin{align}
a(\bm{u},\bm{u})& \ge {\alpha} {E}_{\min} 
C_{K}||\nabla\bm{u}||_{\bm{\mathcal{W}}}^2\quad\forall \bm{u}\in \bm{\mathcal{V}}, 
\label{abound}\\
c(p,p)&\ge  ({\alpha}\beta)^{-1} ||p||^2_{\mathcal{W}} \quad \quad \, \quad\forall  p\in \mathcal{W},
\label{cbound} \\
d(p,p)&\ge  ({\alpha}\beta)^{-1}{E}_{\min} ||p||^2_{\mathcal{W}}\quad\forall  p\in \mathcal{W}, \label{dubound}
\end{align}
where ${0<C_K \leq 1}$ is the Korn constant. In addition, 
there exists a constant $C_{D}>0$ (the inf-sup constant) such that
\begin{align}\label{first_inf_sup}
\sup_{0\neq\bm{v}\in {\bm{\mathcal{V}}}}
\frac{b(\bm{v},q)}{ ||\nabla\bm{v}||_{\bm{\mathcal{W}}} }&\ge C_{D} 
 ||q||_{\mathcal{W}} \quad\forall  q\in \mathcal{W}.
 \end{align}
\end{lemma}
\begin{proof}
{See \cite[Lemma 2.2]{KPSpre}}. 
\end{proof}


Now, to establish that the weak problem is well-posed, and to analyse the error associated with stochastic Galerkin approximations of the weak solution, we will work with a 
$\nu$-dependent norm $||| \cdot |||$ on $\bm{\mathcal{V}}\times\mathcal{W}\times\mathcal{W}$, 
defined by
\begin{align}\label{normden}
|||(\bm{v},q,\tilde{q})|||^2 := {\alpha}||\nabla\bm{v}||_{\bm{\mathcal{W}}}^2+({\alpha}^{-1}+({\alpha}\beta)^{-1})||q||_{\mathcal{W}}^2+({\alpha}\beta)^{-1}||\tilde{q}||_{\mathcal{W}}^2.
\end{align}
The well-posedness of \eqref{scm12} is established in~\cite[Theorem 2.4]{KPSpre}, 
using the following stability result.
\smallskip
 \begin{lemma}\label{welpose11} (\cite[Lemma 2.3]{KPSpre})
 If Assumption \ref{Assump2} holds, then for any 
 $(\bm{u},p,\tilde{p})\in \bm{\mathcal{V}}\times\mathcal{W}\times\mathcal{W}$, there exists a
 $(\bm{v},q,\tilde{q})\in \bm{\mathcal{V}}\times\mathcal{W}\times\mathcal{W}$ 
 with $ |||(\bm{v},q,  \tilde{q} )|||\le C_2 \, |||(\bm{u},p,\tilde{p})|||$, satisfying
 \begin{align}\label{lower_S_bound}
\mathcal{B}(\bm{u},p,\tilde{p}; \bm{v},q,\tilde{q})\ge {E}_{\min}\, C_1 \, |||(\bm{u},p,  \tilde{p})|||^2,
 \end{align}
 where $C_1 :=\frac{1}{2}\min\{C_K,\frac{C_D^2}{E_{\max}^2},\frac{1}{E_{max}^2}\}$ 
 and $C_2:=\sqrt{2(E_{\max}^2+C_D^2+1)}/{E_{\max}}$.
  \end{lemma}


We will make all the constants that appear in our error bounds explicit. Hence, we stress that $C_{K}$ is the Korn constant, $C_{D}$ is the inf-sup constant from \eqref{first_inf_sup} (depending only on $D$), and $E_{\min}$ and $E_{\max}$ are upper and lower bounds for $E$ as defined in \eqref{bounde11}.

\section{Stochastic Galerkin mixed finite element approximation}\label{disformul}

To approximate solutions to \eqref{scm11a}, or equivalently \eqref{scm12}, we begin by choosing finite-dimensional subspaces of $\bm{\mathcal{V}}$ and ${\mathcal{W}}$. Our construction exploits the fact that $\bm{\mathcal{V}}  \cong  \bm{H}_{E_{0}}^{1}(D) \otimes L_{\pi}^{2}(\Gamma)$ and ${\mathcal{W}} \cong L^{2}(D) \otimes L_{\pi}^{2}(\Gamma)$
(i.e., that the spaces are isometrically isomorphic). First, we introduce a mesh ${\mathcal T}_{h}$ on the physical domain $D$ with characteristic mesh size $h$ and choose a pair of conforming finite element spaces
\begin{align*}
V_{h}  = \textrm{span} \left\{ \phi_{r}(\bm{x}) \right \}_{r=1}^{n_{u}}\subset H_{E_{0}}^{1}(D), \qquad
 W_{h}  = \textrm{span} \left\{ \varphi_{s}(\bm{x}) \right \}_{s=1}^{n_{p}} \subset L^{2}(D).
\end{align*} 
We then define $\bm{V}_{h}$ to be the space of vector-valued functions whose components are in $V_{h}$.  We need $\bm{V}_{h}$ and $W_{h}$ to be compatible in the sense that they satisfy the {\it discrete} (deterministic) inf--sup condition 
\begin{align}\label{discrete_inf_sup}
\sup_{0\neq\bm{v}\in {\bm{V}_{h}}}
\frac{\int_D q \,\nabla\cdot\bm{v} }{ ||\nabla\bm{v}||_{\bm{L}^2(D)} }&\ge \gamma_h
 ||q||_{L^2(D)} \quad\forall  q\in W_h,
 \end{align}
with $\gamma_h$ uniformly bounded away from zero (i.e., independent of $h$). Two specific pairs of inf--sup stable spaces $\bm{V}_{h}$--$W_{h}$ on meshes of rectangular elements are $\bm{Q}_2$--${Q}_1$ (continuous biquadratic approximation for the displacement and continuous bilinear approximation for the pressures) and $\bm{Q}_2$--${P}_{-1}$ (continuous biquadratic approximation for the displacement and discontinuous linear approximation for the pressures). 

Next, we consider the parametric discretisation. Let $\{ \psi_{i}(y_{j}), i=0,1, \ldots\}$ denote the set of univariate Legendre polynomials on $\Gamma_{j}  = [-1,1]$, where $ j \in\mathbb{N}$ and $\psi_{i}$ has degree $i$. We fix $\psi_{0}=1$ and assume that the 
polynomials are normalised so that  
$\int_{\Gamma_{j}} \psi_{i}(y_{j}) \psi_{k}(y_{j}) d\pi_{j}(y_{j}) = \delta_{i,k}$.
Now we define the set of finitely supported sequences
$\mathfrak{I} :=\{\bm{\alpha}=(\alpha_1,\alpha_2,\ldots)\in \mathbb{N}_{0}^{\mathbb{N}};\; \# \,\rm{supp}\, \bm{\alpha} <\infty\},$
where ${\rm supp}\, \bm{\alpha} :=\{ m\in\mathbb{N};\, \alpha_m\neq0\}$ for any $\bm{\alpha}\in \mathfrak{I}$.  
The set $\mathfrak{I}$ and its subsets will be called `index sets'.  Combining these ingredients, we can define the countable set of multivariate tensor product polynomials
 \begin{align}
\psi_{\boldsymbol{\alpha}}(\bm{y}): = \prod_{i=1}^{\infty} \psi_{\alpha_{i}}(y_{i}) \quad\forall \bm{\alpha}\in\mathfrak{I}
\end{align}
which forms an orthonormal basis for $L_{\pi}^{2}(\Gamma)$. Now, given any finite index set $ \Lambda \subset \mathfrak{I}$, 
we can define a finite-dimensional subspace of  $L_{\pi}^{2}(\Gamma)$ as follows:
\begin{align}
S_{\Lambda}: =\textrm{span}\left\{ \psi_{\boldsymbol{\alpha}}; \;\; \boldsymbol{\alpha} 
 \in \Lambda \right\}.
\end{align}
Note that only a \emph{finite} number of parameters $y_{k}$ play a role in the definition of $S_{\Lambda}$.

We now define the SG-MFEM approximation spaces
\begin{align*}
\bm{V}_{h,\Lambda}:=\bm{V}_{h} \otimes S_{\Lambda}, \qquad W_{h,\Lambda}:=W_{h} \otimes S_{\Lambda},
\end{align*}
and consider the problem: find
 $({\bm{u}_{h,\Lambda}}, {p_{h,\Lambda}}, {\tilde{p}_{h, \Lambda}})
  \in \bm{V}_{h, \Lambda} \times W_{h,\Lambda} \times W_{h,\Lambda}$ such that
 \begin{subequations} \label{disver11}
\begin{align}
a({\bm{u}_{h,\Lambda}},\bm{v})+b(\bm{v},{p_{h,\Lambda}})&=f(\bm{v}) 
\quad \, \, \forall \bm{v}\in \bm{V}_{h,\Lambda},\\
b({\bm{u}_{h,\Lambda}},q)-c({\tilde{p}_{h, \Lambda}},q)&=0 
\quad \qquad \forall q\in  W_{h,\Lambda},\\
-c({p_{h,\Lambda}},\tilde{q})+d({\tilde{p}_{h, \Lambda}},\tilde{q})&=0 
\quad \qquad \forall \tilde{q}\in  W_{h,\Lambda}.
\end{align}
\end{subequations}
The well-posedness of \eqref{disver11} is established using the following discrete stability result.

\begin{lemma}\label{diswelpose11} 
If $E$ satisfies Assumption \ref{Assump2}, and $\bm{V}_{h}, W_{h}$ satisfy
the inf-sup condition \eqref{discrete_inf_sup} with $\gamma_{h} >0$, then for every
$(\bm{u}_{h,\Lambda},p_{h, \Lambda},\tilde{p}_{h,\Lambda})\in\bm{V}_{h,\Lambda}
 \times {W}_{h,\Lambda} \times {W}_{h, \Lambda}$,
there exists $(\bm{v},q,\tilde{q})\in \bm{V}_{h,\Lambda}  \times {W}_{h,\Lambda} \times {W}_{h, \Lambda}$ 
with
$|||(\bm{v},q,\tilde{q})|||\le C_2^\ast \,  |||(\bm{u}_{h,\Lambda},p_{h,\Lambda},\tilde{p}_{h,\Lambda})|||$,
satisfying 
\begin{align}\label{dislower_S_bound}
\mathcal{B}(\bm{u}_{h,\Lambda},p_{h,\Lambda},\tilde{p}_{h,\Lambda}; 
\bm{v},q,\tilde{q})\ge {E}_{\min}\, C_1^\ast \,
 |||(\bm{u}_{h,\Lambda},p_{h,\Lambda},\tilde{p}_{h,\Lambda})|||^2,
\end{align}
where $C_1^\ast :=\frac{1}{2}\min\{C_K,\frac{\gamma^2_h}{E_{\max}^2},\frac{1}{E_{\max}^2}\}$,
$C_2^\ast :=\sqrt{2(E_{\max}^2+\gamma^2_h+1)}/{E_{\max}}$ and $C_K$ is the Korn constant.
\end{lemma}

\begin{proof}
The proof follows the same lines as that of Lemma \ref{welpose11}, which is given in \cite[Lemma 2.3]{KPSpre}. Note that the constants $C_1^\ast$ and $C_2^\ast$ depend on the \emph{discrete} inf-sup constant $\gamma_{h}$ which is determined by the choice of finite element spaces $\bm{V}_{h}$ and $W_{h}$.
\end{proof}

\smallskip





In the next section, we will outline and analyse a strategy for estimating the SG-MFEM errors $\bm{e}^{\bm{u}}:=\bm{u}-\bm{u}_{h,\Lambda} $, $ e^p:=p-p_{h,\Lambda} $ and $e^{\tilde{p}}:=\tilde{p}-\tilde{p}_{h,\Lambda}$. In order to do this, we will need to introduce spaces that are richer than the chosen $\bm{V}_{h,\Lambda}$ and $W_{h,\Lambda}$. First, let $\bm{V}_h^\ast$ and  $W_h^\ast$ be an inf-sup stable pair (in the sense of \eqref{discrete_inf_sup}) of conforming mixed finite element spaces with $\bm{V}_h\subset \bm{V}_h^\ast$ and $W_{h}\subset W_h^\ast$ such that
\begin{align}
\bm{V}_{h}^\ast= \bm{V}_h\oplus \widetilde{\bm{V}}_h\quad \mbox{and}\quad W_h^\ast=W_h\oplus \widetilde{W}_{h},
\end{align}
where $\widetilde{\bm{V}}_h\subset \bm{H}^{1}_{E_0}(D)$, $\widetilde{W}_h\subset L^2(D)$, and
\begin{align}\label{disjoint}
\bm{V}_h\cap\widetilde{\bm{V}}_h=\{\bm{0}\}, \qquad W_h\cap\widetilde{W}_h=\{{0}\}.
\end{align}  
We will call $\widetilde{\bm{V}}_h$ and $\widetilde{W}_h$ the finite element `detail spaces'. For example, these could be constructed from basis functions that would be introduced by performing a uniform refinement of the mesh associated with $\bm{V}_{h}$ and $W_{h}$. Since $\bm{H}^{1}_{E_0}(D)$ and $L^2(D)$ are Hilbert spaces and \eqref{disjoint} holds, then, by the strengthened Cauchy-Schwarz inequality \cite{MJ, MR1124360,CBS}, there exist constants $\gamma_1, \gamma_2\in [0,1)$ (known as CBS constants) such that 
\begin{subequations} \label{scbic11}
   \begin{align}
      \Big|\int_D \nabla\bm{v}:\nabla\widetilde{\bm{v}} \Big| & \le
      \gamma_1 \, \|\nabla\bm{v}\|_{\bm{L}^{2}(D)}  \|\nabla \widetilde{\bm{v}}\|_{\bm{L}^{2}(D)} \quad  \, 
      \forall \bm{v}\in\bm{V}_{h},\, \forall \widetilde{\bm{v}}\in \widetilde{\bm{V}}_{h},
      \label{scbic11a}
   \\
      \Big|\int_D q\; \widetilde{q} \Big| & \le
      \gamma_2 \, \|q\|_{L^{2}(D)} \, \|\tilde{q}\|_{L^{2}(D)}  \quad \quad  \quad 
      \forall q\in W_{h},\,\forall \widetilde{q}\in \widetilde{W}_{h}.
\label{scbic11b}
\end{align}
\end{subequations}
Next, suppose we choose a new index set  $\mathfrak{Q} \subset \mathfrak{I}$ such that $\Lambda\cap \mathfrak{Q}=\emptyset$ and define ${\Lambda}^\ast:=\Lambda\cup\mathfrak{Q}$. We will call $\mathfrak{Q}$ the `detail' index set. On the parameter domain $\Gamma$, we can then define an enriched polynomial space
 $S_{\Lambda}^{\ast}: = \textrm{span} \left \{ \psi_{\boldsymbol{\alpha}}; \;\; \boldsymbol{\alpha}  \in \Lambda^{\ast} \right\} \subset L_{\pi}^{2}(\Gamma)$.
 Moreover, we have the decomposition
 \begin{align}
S_{{\Lambda}}^\ast= S_{{\Lambda}} \oplus S_{{\mathfrak{Q}}},\quad S_{{\Lambda}} \cap S_{\mathfrak{Q}}=\{0\}.
\end{align}
We can now construct enriched finite-dimensional subspaces of $\bm{\mathcal{V}}$ and $\mathcal{W}$ using the finite element detail spaces $\widetilde{\bm{V}}_h$, 
 $\widetilde{W}_{h}$ and the polynomial space $S_{\mathfrak{Q}}$,
\begin{align}
         \bm{{V}}^{\ast}_{h,\Lambda} \,{:=}\,
         \bm{{V}}_{h^{\ast},\Lambda} \,{\oplus}\, \bm{{V}}_{h,\mathfrak{Q}} \,{=}\,
         \Big(\bm{{V}}_{h,\Lambda} \,{\oplus}\, \widetilde{\bm{{V}}}_{h,\Lambda} \Big) \,{\oplus}\, \bm{{V}}_{h,\mathfrak{Q}} =
         \bm{{V}}_{h,\Lambda} \,{\oplus}\, \Big(\widetilde{\bm{{V}}}_{h,\Lambda} \,{\oplus}\, \bm{{V}}_{h,\mathfrak{Q}} \Big),
         \label{enriched_spacesV}
         \\
         {{W}}^{\ast}_{h,\Lambda} \,{:=}\,
         {{W}}_{h^{\ast},\Lambda} \,{\oplus}\, {{W}}_{h,\mathfrak{Q}} \,{=}\,
         \Big(W_{h,\Lambda} \,{\oplus}\, \widetilde{{W}}_{h,\Lambda}\Big) \,{\oplus}\, {{W}}_{h,\mathfrak{Q}} =
         {{W}}_{h,\Lambda} \,{\oplus}\, \Big( \widetilde{{W}}_{h,\Lambda} \,{\oplus}\, {{W}}_{h,\mathfrak{Q}} \Big)
         \label{enriched_spacesW},
\end{align}
where $\bm{V}_{h^\ast,\Lambda}:=\bm{V}_h^\ast \otimes S_\Lambda$, $W_{h^\ast,\Lambda}:=W_h^\ast \otimes S_\Lambda$, $\widetilde{\bm{V}}_{h,\Lambda}:=\widetilde{\bm{V}}_h \otimes S_\Lambda$, $\widetilde{W}_{h,\Lambda}:=\widetilde{W}_{h}\otimes S_\Lambda$, $\bm{V}_{h,\mathfrak{Q}}=\bm{V}_{h} \otimes S_{\mathfrak{Q}}$ and $W_{h,\mathfrak{Q}}=W_{h} \otimes S_{\mathfrak{Q}}.$


%

\section{A posteriori error estimation}\label{sec4}
We now want to estimate the SG-MFEM errors $\bm{e}^{\bm{u}}=\bm{u}-\bm{u}_{h,\Lambda}$,
$e^p=p-p_{h,\Lambda}$ and $e^{\tilde{p}}=\tilde{p}-\tilde{p}_{h,\Lambda}$. To that end, we adapt the \emph{residual approach} described, e.g., in~\cite[Section~3]{prudhommeOden1999} to our three-field formulation. Substituting  $\bm{u}$, ${p}$ and $\tilde{p}$ in \eqref{scm11a} by
$\bm{u}_{h,\Lambda}+e^{\bm{u}}$, $p_{h,\Lambda}+e^p$ and $\tilde{p}_{h,\Lambda}+e^{\tilde{p}}$, respectively,
we conclude that $(e^{\bm{u}}, e^p, e^{\tilde{p}}) \in \bm{\mathcal{V}} \times \mathcal{W} \times \mathcal{W}$ satisfies
\begin{subequations}\label{scm11aer}
\begin{align}
a(e^{\bm{u}},\bm{v})+b(\bm{v},e^p)&=\mathcal{R}^{\bm{u}}(\bm{v})\quad\forall \bm{v}\in \bm{\mathcal{V}},\\
b(e^{\bm{u}},q)-c(e^{\tilde{p}},q)&=\mathcal{R}^{p}(q)\quad \, \, \forall q\in \mathcal{W},\\
-c(e^p,\tilde{q})+d(e^{\tilde{p}},\tilde{q})&=\mathcal{R}^{\tilde{p}}(\tilde{q})\quad \, \, \forall \tilde{q}\in \mathcal{W},
\end{align}
\end{subequations}
where the linear functionals $\mathcal{R}^{\bm{u}}:\bm{\mathcal{V}}\rightarrow \mathbb{R}$ and 
$\mathcal{R}^{p},\mathcal{R}^{\tilde{p}}:\mathcal{W}\rightarrow \mathbb{R}$ are defined as
\begin{subequations}\label{resid}
   \begin{align}
       \mathcal{R}^{\bm{u}}(\bm{v})
       &:=f(\bm{v})-a(\bm{u}_{h,\Lambda},\bm{v})-b(\bm{v},p_{h,\Lambda})\quad \, \forall \bm{v}\in \bm{\mathcal{V}},
       \label{resid1}
   \\
       \mathcal{R}^{p}(q)
       &:=-b(\bm{u}_{h,\Lambda},q)+c(\tilde{p}_{h,\Lambda},q)\quad \quad \quad \quad \forall q\in \mathcal{W},
       \label{resid2}
   \\
       \mathcal{R}^{\tilde{p}}(\tilde{q})
       &:=c(p_{h,\Lambda},\tilde{q})-d(\tilde{p}_{h,\Lambda},\tilde{q}) \qquad \qquad \quad \forall \tilde{q}\in \mathcal{W}.
       \label{resid3}
   \end{align}
\end{subequations}
Clearly, these functionals represent the residuals  associated with the current SG-MFEM approximation.
For each one, we define a weighted dual norm as follows 
\begin{align}
||\mathcal{R}^{\bm{u}}||_\ast &:= \sup_{\bm{v}\in\bm{\mathcal{V}}\setminus\{0\}}\frac{|\mathcal{R}^{\bm{u}}(\bm{v})|}{\alpha^{1/2}||\nabla\bm{v}||_{\bm{\mathcal{W}}}},\\
 ||\mathcal{R}^{p}||_\ast &:= \sup_{q\in{\mathcal{W}}\setminus\{0\}}\frac{|\mathcal{R}^{p}(q)|}{(\alpha^{-1}+(\alpha\beta)^{-1})^{1/2}||q||_{{\mathcal{W}}}}, \\
  ||\mathcal{R}^{\tilde{p}}||_\ast &:= \sup_{\tilde{q}\in{\mathcal{W}}\setminus\{0\}}\frac{|\mathcal{R}^{\tilde{p}}(\tilde{q})|}{(\alpha\beta)^{-1/2}||\tilde{q}||_{{\mathcal{W}}}}.
\end{align}
The next result establishes an equivalence between the norm of the SG-MFEM approximation error and the sum of the dual norms of the three residuals. 

\begin{theorem}\label{mainth1}
Let $(e^{\bm{u}},e^{p},e^{\tilde{p}}) \in \bm{\mathcal V} \times {\mathcal{W}} \times {\mathcal W}$ be the error in the SG-MFEM approximation $(\bm{u}_{h,\Lambda}, p_{h,\Lambda}, \tilde{p}_{h,\Lambda}) \in \bm{V}_{h,\Lambda} \times {W}_{h,\Lambda} \times {W}_{h,\Lambda}$ of the solution to (\ref{scm12}). Then
\begin{align*} 
C_6\, (||\mathcal{R}^{\bm{u}}||_\ast+||\mathcal{R}^{p}||_\ast+||\mathcal{R}^{\tilde{p}}||_\ast)\le |||(e^{\bm{u}},e^{p},e^{\tilde{p}})|||\le C_7 \, (||\mathcal{R}^{\bm{u}}||_\ast+||\mathcal{R}^{p}||_\ast+||\mathcal{R}^{\tilde{p}}||_\ast),
\end{align*}
where
$C_6:=\big( \sqrt{3} \max\big\{E_{\max} + \sqrt{d},\, 1+\sqrt{d}\big\} \big)^{-1}$, $C_7:=C_2/(C_1E_{\min})$ and $C_1$ and $C_2$ are the constants defined in Lemma \ref{welpose11}.
\end{theorem}
\smallskip
\begin{proof}
Since $(e^{\bm{u}},e^{p},e^{\tilde{p}})\in \bm{\mathcal{V}}\times\mathcal{W}\times\mathcal{W}$, then from Lemma \ref{welpose11} there exists a $(\bm{v},q,\tilde{q})\in \bm{\mathcal{V}}\times\mathcal{W}\times\mathcal{W}$ with $|||(\bm{v},q,\tilde{q})|||\le C_2 \, |||(e^{\bm{u}},e^{p},e^{\tilde{p}})|||$ such that
\begin{align}
C_1 \, E_{\min} \, |||(e^{\bm{u}},e^{p},e^{\tilde{p}})|||^2\le \mathcal{B}(e^{\bm{u}},e^{p},e^{\tilde{p}}; \bm{v},q,\tilde{q}),
\end{align}
where $C_1$ and $C_2$ depend on $E_{\max}$, $C_K$ and $C_D$.
Hence, we have by (\ref{scm11aer})
\begin{align}
C_1 \, E_{\min} \, |||(e^{\bm{u}},e^{p},e^{\tilde{p}})|||^2\le \mathcal{R}^{\bm{u}}(\bm{v})+\mathcal{R}^{p}(q)+\mathcal{R}^{\tilde{p}}(\tilde{q}).
\end{align}
Moreover, using the definitions of the dual norms, we have
\begin{align}
C_1 E_{\min} |||(e^{\bm{u}},e^{p},e^{\tilde{p}})|||^2\le (||\mathcal{R}_h^{\bm{u}}||_{\ast}+||\mathcal{R}_h^{p}||_{\ast}+||\mathcal{R}_h^{\tilde{p}}||_{\ast})C_2|||(e^{\bm{u}},e^{p},e^{\tilde{p}})|||.
\end{align}
This establishes the upper bound. To establish the lower bound, we use the definition
\begin{align}
||\mathcal{R}^{\bm{u}}||_\ast &:= \sup_{\bm{v}\in\bm{\mathcal{V}}\setminus\{0\}}\frac{|\mathcal{R}^{\bm{u}}(\bm{v})|}{\alpha^{1/2}||\nabla\bm{v}||_{\bm{\mathcal{W}}}}=\sup_{\bm{v}\in\bm{\mathcal{V}}\setminus\{0\}}\frac{|a(e_{\bm{u}},\bm{v})+b(\bm{v},e_p)|}{\alpha^{1/2}||\nabla\bm{v}||_{\bm{\mathcal{W}}}},
\end{align} 
and apply Lemma \ref{tecre11} to give
\begin{align}\label{errre11}
||\mathcal{R}^{\bm{u}}||_\ast \le E_{\max} \alpha^{1/2}||\nabla e^{\bm{u}}||_{\bm{\mathcal{W}}}+\sqrt{d}\alpha^{-1/2}||e^p||_{\mathcal{W}}.
\end{align}
Similarly, we have
\begin{align}\label{errre12}
||\mathcal{R}^{p}||_{\ast}&\le \sqrt{d} \alpha^{1/2}||\nabla e^{\bm{u}}||_{\bm{\mathcal{W}}}+(\alpha\beta)^{-1/2}||e^{\tilde{p}}||_{\mathcal{W}},\\
||\mathcal{R}^{\tilde{p}}||_{\ast}&\le\ (\alpha\beta)^{-1/2}||e^p||_{\mathcal{W}}+E_{\max}(\alpha\beta)^{-1/2}||e^{\tilde{p}}||_{\mathcal{W}}.\label{errre13}
\end{align}
Combining (\ref{errre11}), (\ref{errre12}) and (\ref{errre13}), and recalling the definition of the norm $|||(\cdot,\cdot,\cdot)|||$ in (\ref{normden}) implies the stated result. 
\end{proof}

Theorem \ref{mainth1} is our starting point for developing an a posteriori error estimation strategy. We will estimate $|||(e^{\bm{u}},e^{p},e^{\tilde{p}})|||$ by estimating $||\mathcal{R}^{\bm{u}}||_\ast+||\mathcal{R}^{p}||_\ast+||\mathcal{R}^{\tilde{p}}||_\ast $.

\subsection{Evaluation of the residuals}
First, we give an alternative representation of each of $||\mathcal{R}^{\bm{u}}||_\ast$, $||\mathcal{R}^{p}||_\ast$ and $||\mathcal{R}^{\tilde{p}}||_\ast$ and show that  $||\mathcal{R}^{p}||_\ast$ can be evaluated exactly. Since $\bm{\mathcal{V}}$ and $\mathcal{W}$ are Hilbert spaces, the Riesz representation theorem tells us that we can find a unique ${\bm{e}}_{0}^{\bm{u}}\in \bm{\mathcal{V}}$, a unique $e_{0}^{p}\in\mathcal{W}$ and a unique $e_{0}^{\tilde{p}}\in\mathcal{W}$ such that 
\begin{subequations}\label{exev1}
\begin{align}\label{res11}
\bar{a}_0(\bm{e}_{0}^{\bm{u}},\bm{v})&=\mathcal{R}^{\bm{u}}(\bm{v}) \qquad  \forall \bm{v}\in \bm{\mathcal{V}},\\
\bar{c}({e}_{0}^{p},q)&=\mathcal{R}^p(q) \,  \qquad \forall q\in {\mathcal{W}},\label{res12}\\
\bar{d}_0({e}_{0}^{\tilde{p}},\tilde{q})&=\mathcal{R}^{\tilde{p}}(\tilde{q}) \, \qquad \forall \tilde{q}\in {\mathcal{W}},\label{res13}
\end{align} 
\end{subequations}
where we define the weighted inner products
\begin{align}
\bar{a}_0(\bm{u},\bm{v})&:= \alpha \int_{\Gamma}\int_{D} \nabla \bm{u}(\bm{x},\bm{y}):\nabla\bm{v}(\bm{x},\bm{y}) \, d\bm{x} \, d\pi(\bm{y}),\\
\bar{c}(p,q)&:=\Big(\alpha^{-1}+(\alpha\beta)^{-1}\Big)\int_{\Gamma} \int_{D} {p}(\bm{x},\bm{y}){q}(\bm{x},\bm{y}) \, d\bm{x} \, d\pi(\bm{y}),\\
\bar{d}_0(p,q)&:=({\alpha}\beta)^{-1}\int_{\Gamma}\int_{D} {p}(\bm{x},\bm{y}){q}(\bm{x},\bm{y}) \, d\bm{x} \, d\pi(\bm{y}).
\end{align}
{Note that $\bar{d}_0(\cdot, \cdot)$ is identical to $c(\cdot, \cdot)$ in~\eqref{fbf_c}.
 We use the notation $\bar{d}_0(\cdot, \cdot)$ to emphasise that it is an approximation to the bilinear form $d(\cdot, \cdot)$ defined in \eqref{fbf_d}. Similarly,  $\bar{a}_0(\cdot, \cdot)$ approximates the bilinear form $a(\cdot, \cdot)$ in \eqref{fbf_a}. We observe that 
\begin{align}
|\bm{e}_{0}^{\bm{u}}|_{\bar{a}_0}&:=\sqrt{\bar{a}_0(\bm{e}_{0}^{\bm{u}},\bm{e}_{0}^{\bm{u}})}=||\mathcal{R}^{\bm{u}}||_{\ast}, \label{duales1}\\
|{e}_{0}^{p}|_{\bar{c}}&:=\sqrt{\bar{c}({e}_{0}^{p},{e}_{0}^{p})}=||\mathcal{R}^p||_\ast,  \label{duales2}\\
|{e}_{0}^{\tilde{p}}|_{\bar{d}_0}&:=\sqrt{\bar{d}_0({e}_{0}^{\tilde{p}},{e}_{0}^{\tilde{p}})}=||\mathcal{R}^{\tilde{p}}||_\ast. \label{duales3}
\end{align} 
Now, from (\ref{res12}) and the definition of $\mathcal{R}^p$ in~\eqref{resid2},  $\bar{c}(e_{0}^{p},q)=-b(\bm{u}_{h,\Lambda},q) + c(\tilde{p}_{h,\Lambda},q)$
for all $q \in {\mathcal{W}}$, from which we conclude that
\begin{align}\label{e_0_p_def}
{e}_{0}^p= (\alpha^{-1}+(\alpha\beta)^{-1})^{-1}(\nabla \cdot \bm{u}_{h,\Lambda}+(\alpha\beta)^{-1} \tilde{p}_{h,\Lambda}) \quad \mbox{a.e. in}\; D\times \Gamma. 
\end{align}  
Hence,  unlike  $||\mathcal{R}^{\bm{u}}||_{\ast}$ and $||\mathcal{R}^{\tilde{p}}||_\ast$, we can compute $||\mathcal{R}^p||_\ast = |{e}_{0}^{p}|_{\bar{c}}$ directly. This is due to the fact that $\mathcal{R}^p$ only involves the bilinear forms $b(\cdot,\cdot)$ and $c(\cdot,\cdot)$, neither of which include the parametric Young's modulus $E$.

\subsection{Approximation of the residuals}
We can approximate the solutions $\bm{e}_{0}^{\bm{u}}$ and ${e}_{0}^{\tilde{p}}$ to \eqref{res11} and \eqref{res13} by replacing the infinite-dimensional spaces $\bm{\mathcal V}$ and ${\mathcal W}$  with finite-dimensional ones. However, since $\mathcal{R}^{\bm{u}}(\bm{v})=0$ for all ${\bm{v}} \in \bm{V}_{h,\Lambda}$ and $\mathcal{R}^{\tilde{p}}(\tilde{q}) =0$ for all $\tilde{q} \in {W}_{h,\Lambda}$, we must work with richer spaces than $\bm{V}_{h,\Lambda}$ and ${W}_{h,\Lambda}$. 

Recall now the definitions of the spaces ${\bm{V}}_{h,\Lambda}^{\ast}\subset \bm{\mathcal{V}}$ and ${W}_{h,\Lambda}^{\ast} \subset {\mathcal{W}}$  in \eqref{enriched_spacesV}  and \eqref{enriched_spacesW} and consider the problems: find $\bm{e}_{0}^{\bm{u},\ast}\in{\bm{V}}_{h,\Lambda}^\ast$ and ${e}_{0}^{\tilde{p},\ast}\in{W}_{h,\Lambda}^\ast$ such that
\begin{subequations}\label{disev1}
\begin{align}\label{disres11}
\bar{a}_0(\bm{e}_{0}^{\bm{u},\ast},\bm{v})&=\mathcal{R}^{\bm{u}}(\bm{v}) \qquad  \forall \bm{v}\in {\bm{V}}_{h,\Lambda}^\ast,\\
\bar{d}_0({e}_{0}^{\tilde{p},\ast},\tilde{q})&=\mathcal{R}^{\tilde{p}}(\tilde{q}) \qquad \forall \tilde{q}\in {{W}}_{h,\Lambda}^\ast.\label{disres13}
\end{align} 
\end{subequations}
From  (\ref{res11}) and (\ref{disres11}), we have
\[
   \bar{a}_0(\bm{e}_{0}^{\bm{u}}-\bm{e}_{0}^{\bm{u},\ast},\bm{v}) = 0\quad \forall \bm{v}\in \bm{V}_{h,\Lambda}^\ast,
\]
and similarly, 
\[
   \bar{d}_0({e}_{0}^{\tilde{p}}-{e}_{0}^{\tilde{p},\ast},\tilde{q}) = 0 \qquad \forall \tilde{q}\in {{W}}_{h,\Lambda}^\ast.
\]
\newpage
Hence, $\bm{e}_{0}^{\bm{u},\ast}$ is simply the orthogonal projection of $\bm{e}_{0}^{\bm{u}}$
onto the space ${\bm{V}}_{h,\Lambda}^\ast$ with respect to the inner product $\bar{a}_0({\cdot,\cdot})$, and
${e}_{0}^{\tilde{p},\ast}$ is the orthogonal projection of ${e}_{0}^{\tilde{p}}$
onto the space $W_{h,\Lambda}$ with respect to $\bar{d}_0(\cdot,\cdot)$.
As a consequence of this, we have
\begin{align}\label{oreq1}
|\bm{e}_{0}^{\bm{u},\ast}|^2_{\bar{a}_0}+ |\bm{e}_{0}^{\bm{u}}-\bm{e}_{0}^{\bm{u},\ast}|^2_{\bar{a}_0}= |\bm{e}_{0}^{\bm{u}}|^2_{\bar{a}_0}, \quad 
|{e}_{0}^{\tilde{p},\ast}|^2_{\bar{d}_0}+ |{e}_{0}^{\tilde{p}}-{e}_{0}^{\tilde{p},\ast}|^2_{\bar{d}_0}= |{e}_{0}^{\tilde{p}}|^2_{\bar{d}_0},
\end{align}
which gives
\begin{align}\label{ortest1}
|\bm{e}_{0}^{\bm{u}}-\bm{e}_{0}^{\bm{u},\ast}|^2_{\bar{a}_0}\le |\bm{e}_{0}^{\bm{u}}|^2_{\bar{a}_0},
\quad  |{e}_{0}^{\tilde{p}}-{e}_{0}^{\tilde{p},\ast}|^2_{\bar{d}_0}\le |{e}_{0}^{\tilde{p}}|^2_{\bar{d}_0}.
\end{align}


\begin{lemma}\label{mainth2}
Let $\bm{e}_{0}^{\bm{u}} \in \bm{\mathcal{V}}$ and ${e}_{0}^{\tilde{p}}\in \mathcal{W}$ satisfy  \eqref{res11} and \eqref{res13}, respectively, and let $\bm{e}_{0}^{\bm{u},\ast} \in  \bm{V}_{h,\Lambda}^\ast$ and ${e}_{0}^{\tilde{p},\ast} \in W_{h,\Lambda}^\ast$ {satisfy \eqref{disres11} and \eqref{disres13}, respectively}. Suppose that $\bm{e}_{0}^{\bm{u}}\neq0$ and ${e}_{0}^{\tilde{p}}\neq 0$ and the spaces $\bm{V}_{h,\Lambda}^\ast$ and $W_{h,\Lambda}^\ast$ are chosen so that $\bm{e}_{0}^{\bm{u},\ast}\neq0$ and ${e}_{0}^{\tilde{p},\ast}\neq 0$. Then, there exists a constant $\Theta\in[0,1)$ such that:
\begin{subequations}\label{ortest3}
\begin{align}
| \bm{e}_{0}^{\bm{u},\ast}|_{\bar{a}_0} &\le | \bm{e}_{0}^{\bm{u}}|_{\bar{a}_0}\le \frac{1}{\sqrt{1-\Theta^2}}  |\bm{e}_{0}^{\bm{u},\ast}|_{\bar{a}_0},\\
| e_{0}^{\tilde{p},\ast}|_{\bar{d}_0} &\le | e_{0}^{\tilde{p}}|_{\bar{d}_0}\le \frac{1}{\sqrt{1-\Theta^2}}  |e_{0}^{\tilde{p},\ast}|_{\bar{d}_0}.
\end{align}
\end{subequations}
\end{lemma}
\begin{proof}The left hand-side inequalities in (\ref{ortest3}) follow directly from (\ref{oreq1}).  If $\bm{e}_{0}^{\bm{u},\ast}\neq0$ and ${e}_{0}^{\tilde{p},\ast}\neq 0$ then (\ref{oreq1}) also tells us that there exists a $\Theta\in[0,1)$ such that:
\begin{align}\label{ortest2}
|\bm{e}_{0}^{\bm{u}}-\bm{e}_{0}^{\bm{u},\ast}|_{\bar{a}_0} \le \Theta |\bm{e}_{0}^{\bm{u}}|_{\bar{a}_0}, \quad 
|{e}_{0}^{\tilde{p}}-{e}_{0}^{\tilde{p},\ast}|_{\bar{d}_0} \le \Theta |{e}_{0}^{\tilde{p}}|_{\bar{d}_0}.
\end{align}
Using (\ref{oreq1}) and  squaring both sides of the equations in (\ref{ortest2}) gives
\begin{subequations}\label{relexdis1}
\begin{align}
|\bm{e}_{0}^{\bm{u}}|_{\bar{a}_0}^2-|\bm{e}_{0}^{\bm{u},\ast}|_{\bar{a}_0}^2  =  |\bm{e}_{0}^{\bm{u}}-\bm{e}_{0}^{\bm{u},\ast}|_{\bar{a}_0}^2&\le \Theta^2 |\bm{e}_{0}^{\bm{u}}|_{\bar{a}_0}^2,\\
|{e}_{0}^{\tilde{p}}|_{\bar{d}_0}^2-|{e}_{0}^{\tilde{p},\ast}|_{\bar{d}_0}^2 = |{e}_{0}^{\tilde{p}}-{e}_{0}^{\tilde{p},\ast}|_{\bar{d}_0}^2&\le \Theta^2 |{e}_{0}^{\tilde{p}}|_{\bar{d}_0}^2,
\end{align}
\end{subequations}
which proves the right hand-side inequalities in (\ref{ortest3}).
\end{proof}

We can interpret $\Theta$ {in the above result} as a saturation constant (cf.~\cite[Section~3.1]{prudhommeOden1999} and~\cite[Chapter~5]{MJ}). Indeed, if $\bm{V}_{h,\Lambda}^\ast$ and $W_{h,\Lambda}^\ast$ are not informative enough, then we will obtain $\bm{e}_{0}^{\bm{u},\ast} \approx 0$ and ${e}_{0}^{\tilde{p},\ast} \approx 0$ and thus $\Theta \approx 1$.
However, if the spaces are sufficiently rich, then $\bm{e}_{0}^{\bm{u},\ast}$ (resp., ${e}_{0}^{\tilde{p},\ast}$) will be a good estimator for $\bm{e}_{0}^{\bm{u}}$ (resp., ${e}_{0}^{\tilde{p}}$) and $\Theta$ will be close to zero. 



Although the estimates $|\bm{e}_{0}^{\bm{u},\ast}|_{\bar{a}_0}$ and $|{e}_{0}^{\tilde{p},\ast}|_{\bar{d}_0}$ are computable, the dimensions of $\bm{V}_{h,\Lambda}^\ast$ and $ W_{h,\Lambda}^\ast$ may be much larger than those of $\bm{V}_{h,\Lambda}$ and $ W_{h,\Lambda}$. Fortunately, we can exploit the structure shown in \eqref{enriched_spacesV} and \eqref{enriched_spacesW} to obtain lower-dimensional problems, leading to estimates that are cheaper to compute. A suitable strategy for estimating the energy error for scalar diffusion problems was developed in~\cite{bps14}. In a similar manner, instead of solving \eqref{disres11} and \eqref{disres13}, we consider the {alternative} problems: 
find $\bm{e}_{2}^{\bm{u}} \in \bm{V}_{2}$ and ${e}_{2}^{\tilde{p}} \in W_{2} $ such that
\begin{subequations}\label{disev1_cp}
\begin{align}\label{disres11_cp}
\bar{a}_0(\bm{e}_{2}^{\bm{u}},\bm{v})&=\mathcal{R}^{\bm{u}}(\bm{v}) \qquad  \forall \bm{v}\in {\bm{V}}_{2},\\
\bar{d}_0({e}_{2}^{\tilde{p}},\tilde{q})&=\mathcal{R}^{\tilde{p}}(\tilde{q}) \qquad \, \forall \tilde{q}\in W_{2}, \label{disres13_cp}
\end{align} 
\end{subequations}
where we now define
$$\bm{V}_{2}:= \widetilde{\bm{V}}_{h, \Lambda} \oplus \bm{V}_{h, \mathfrak{Q}}, \qquad W_{2}:=\widetilde{W}_{h,\Lambda} \oplus W_{h,\mathfrak{Q}}. $$
Note that this means that we only solve the problems on the spaces that have been added to  $\bm{V}_{h,\Lambda}$ and $W_{h,\Lambda}$ {to form $\bm{V}_{h,\Lambda}^\ast$ and $ W_{h,\Lambda}^\ast$}. We then define the error estimates
\begin{align} \label{err:estimates}
\eta_{1}:=|\bm{e}_{2}^{\bm{u}} |_{\bar{a}_{0}},  \quad \eta_2 := |{e}_{0}^p|_{\bar{c}}, \quad  \eta_{3}:=|e_{2}^{\tilde{p}} |_{\bar{d}_{0}},
\end{align}
where we recall that $e_{0}^{p}$ is given {directly} by \eqref{e_0_p_def}.

Now, we have the {decompositions}
\begin{align}
\bm{e}_{2}^{\bm{u}} = \tilde{\bm{e}}_{h,\Lambda}^{\bm{u}} + \bm{e}_{h,\mathfrak{Q}}^{\bm{u}}, \quad
{e}_{2}^{\tilde{p}} = \tilde{e}_{h,\Lambda}^{\tilde{p}} + {e}_{h,\mathfrak{Q}}^{\tilde{p}},
\end{align}
where $\tilde{\bm{e}}_{h,\Lambda}^{\bm{u}}\in \widetilde{\bm{V}}_{h, \Lambda},$ $\bm{e}_{h,\mathfrak{Q}}^{\bm{u}}\in \bm{V}_{h,\mathfrak{Q}}$, $\tilde{{e}}_{h,\Lambda}^{\tilde{p}}\in \widetilde{W}_{h,\Lambda},$  and ${e}_{h,\mathfrak{Q}}^{\tilde{p}}\in W_{h, \mathfrak{Q}}$. Choosing test functions $\bm{v}\in \widetilde{{\bm{V}}}_{h,\Lambda}$ in \eqref{disres11_cp} and $\tilde{q} \in \widetilde{{W}}_{h,\Lambda}$ in \eqref{disres13_cp} and using the orthogonality of the polynomials in $S_{\Lambda}$ and $S_{\mathfrak{Q}}$  with respect to {the measure} $\pi$ gives 
\begin{subequations} \label{cdisressp}
\begin{align}\label{cdisressp11}
\bar{a}_0(\tilde{\bm{e}}_{h,\Lambda}^{\bm{u}},\bm{v})&=\mathcal{R}^{\bm{u}}(\bm{v}) \quad\forall \bm{v}\in \widetilde{\bm{V}}_{h,\Lambda},\\
\bar{d}_0(\tilde{e}_{h,\Lambda}^{\tilde{p}},\tilde{q})&=\mathcal{R}^{\tilde{p}}(\tilde{q}) \quad \, \forall \tilde{q}\in \widetilde{{W}}_{h,\Lambda}. \label{cdisressp13}
\end{align}
\end{subequations}
We will refer to $\tilde{\bm{e}}_{h,\Lambda}^{\bm{u}}$ and $\tilde{e}_{h,\Lambda}^{\tilde{p}}$ as the \emph{spatial} error estimators.
Similarly, choosing test functions $\bm{v}\in {\bm{V}}_{h,\mathfrak{Q}}$ in \eqref{disres11_cp} and $\tilde{q} \in {W}_{h,\mathfrak{Q}}$ in \eqref{disres13_cp} gives
\begin{subequations} \label{disrespar}
\begin{align}\label{disrespar11}
\bar{a}_0(\bm{e}_{h,\mathfrak{Q}}^{\bm{u}},\bm{v})&=\mathcal{R}^{\bm{u}}(\bm{v}) \quad\forall \bm{v}\in {\bm{V}}_{h,\mathfrak{Q}},\\
\bar{d}_0({e}_{h,\mathfrak{Q}}^{\tilde{p}},\tilde{q})&=\mathcal{R}^{\tilde{p}}(\tilde{q}) \quad \,\forall \tilde{q}\in {{W}}_{h,\mathfrak{Q}}. \label{disrespar13}
\end{align}
\end{subequations}
We will refer to $\bm{e}_{h,\mathfrak{Q}}^{\bm{u}}$ and ${e}_{h,\mathfrak{Q}}^{\tilde{p}}$ as the \emph{parametric} error estimators.
 
Since $\bar{a}_{0}(\tilde{\bm{e}}_{h,\Lambda}^{\bm{u}}, \bm{e}_{h,\mathfrak{Q}}^{\bm{u}})= 0$ and
$\bar{d}_0(\tilde{e}_{h,\Lambda}^{\tilde{p}}, {e}_{h,\mathfrak{Q}}^{\tilde{p}})= 0$, 
we can write $\eta_{1}$ and $\eta_{3}$ as
\begin{align} \label{eta1_Def}
   \eta_1 := |\bm{e}_{2}^{\bm{u}}|_{\bar{a}_0} & =
   |\tilde{\bm{e}}_{h,\Lambda}^{\bm{u}} + \bm{e}_{h,\mathfrak{Q}}^{\bm{u}}|_{\bar{a}_0} = 
   (|\tilde{\bm{e}}_{h,\Lambda}^{\bm{u}}|_{\bar{a}_0}^2+|\bm{e}_{h,\mathfrak{Q}}^{\bm{u}}|_{\bar{a}_0}^2)^{1/2},\\
\label{eta3_Def}\eta_3 := |e_{2}^{\tilde{p}}| _{\bar{d}_0}& =  |\tilde{e}_{h,\Lambda}^{\tilde{p}}+{e}_{h,\mathfrak{Q}}^{\tilde{p}}|_{\bar{d}_0} = (|\tilde{e}_{h,\Lambda}^{\tilde{p}}|_{\bar{d}_0}^2+|{e}_{h,\mathfrak{Q}}^{\tilde{p}}|_{\bar{d}_0}^2)^{1/2}.
\end{align}
Hence, to compute $\eta_{1}$ and $\eta_{3}$, we solve the four decoupled problems \eqref{cdisressp11}--\eqref{disrespar13}. The next result characterises how well $\eta_{1}$ and $\eta_{3}$ approximate $ | \bm{e}_{0}^{\bm{u},\ast}|_{\bar{a}_0}$ and $ |  e_{0}^{\tilde{p},\ast}|_{\bar{d}_0}$, respectively. 

\begin{lemma}\label{mainlam11}
Let $ e_{0}^{\bm{u},\ast} \in \bm{V}_{h,\Lambda}^{\ast}$ and $ e_{0}^{\tilde{p},\ast} \in W_{h,\Lambda}^{\ast}$ satisfy (\ref{disev1}). Let the error estimates $\eta_1$ and $\eta_3$ be defined as in \eqref{eta1_Def} and \eqref{eta3_Def}. Then,
\begin{align} \label{eq:mainlam11}
\eta_1&\le | \bm{e}_{0}^{\bm{u},\ast}|_{\bar{a}_0} \le \frac{1}{\sqrt{1-\gamma_{1}^2}} \, \eta_1, \qquad 
\eta_3 \le |  e_{0}^{\tilde{p},\ast}|_{\bar{d}_0} \le \frac{1}{\sqrt{1-\gamma_{2}^2}} \, \eta_3,
\end{align}
where 
$\gamma_1$ and $\gamma_2$ are the CBS constants 
{from \eqref{scbic11a} and \eqref{scbic11b} respectively}.
\end{lemma}

\begin{proof}
Since $ \bm{e}_{2}^{\bm{u}} \in \bm{V}_{2}$ and $\bm{V}_{2} \subset \bm{V}_{h, \Lambda}^{\ast}$, combining \eqref{disres11} and \eqref{disres11_cp} gives
\begin{align*}
\bar{a}_{0}(\bm{e}_{0}^{\bm{u},\ast}, \bm{e}_{2}^{\bm{u}}) = \bar{a}_{0}(\bm{e}_{2}^{\bm{u}}, \bm{e}_{2}^{\bm{u}}) = \eta_{1}^{2}.
\end{align*} 
Applying the Cauchy-Schwarz inequality then gives $\eta_{1} \le  |\bm{e}_{0}^{\bm{u},\ast}|_{\bar{a}_{0}}$. Similarly, combining \eqref{disres13} and \eqref{disres13_cp}  gives  $\eta_{3} \le  | e_{0}^{\tilde{p},\ast}|_{\bar{d}_{0}}$.

To prove the upper bounds, we decompose $\bm{e}_{0}^{\bm{u},\ast}\in \bm{V}_{h,\Lambda}^\ast$ and ${e}_{0}^{\tilde{p},\ast}\in W_{h,\Lambda}^\ast$ as follows:
\bse \label{err:decomp}
\bea
         \bm{e}_{0}^{\bm{u},\ast} = \bm{r}_{h,\Lambda}^{\bm{u}} + \bm{r}_{2}^{\bm{u}}\quad
         & \rm with \quad &
         \bm{r}_{2}^{\bm{u}} = \tilde{\bm{r}}_{h,\Lambda}^{\bm{u}}+\bm{r}_{h,\mathfrak{Q}}^{\bm{u}}\;  \in \bm{V}_2,
         \label{err:decomp:1}
         \\
         {e}_{0}^{\tilde{p},\ast} = r_{h,\Lambda}^{\tilde{p}} + r_{2}^{\tilde{p}}\quad
         & \rm with \quad &
         r_{2}^{\tilde{p}} = \tilde{r}_{h,\Lambda}^{\tilde{p}}+r_{h,\mathfrak{Q}}^{\tilde{p}}\; \in W_2.
         \label{err:decomp:2}
\eea
\ese
Here, $\bm{r}_{h,\Lambda}^{\bm{u}}\in\bm{V}_{h,\Lambda}$, $\tilde{\bm{r}}_{h,\Lambda}^{\bm{u}}\in\widetilde{\bm{V}}_{h,\Lambda}$, $\bm{r}_{h,\mathfrak{Q}}^{\bm{u}}\in {\bm{V}}_{h,\mathfrak{Q}}$, $r_{h,\Lambda}^{\tilde p}\in W_{h,\Lambda}$, $\tilde{r}_{h,\Lambda}^{\tilde{p}}\in \widetilde{W}_{h,\Lambda}$ and $r_{h,\mathfrak{Q}}^{\tilde{p}}\in {W}_{h,\mathfrak{Q}}$. Recall that $\bm{V}_{h,\Lambda}$ is a subspace of $\bm{V}^\ast_{h,\Lambda}$ and $W_{h,\Lambda}$ is a subspace of $W^\ast_{h,\Lambda}$.
Hence, using the definitions of the residuals \eqref{resid1} and \eqref{resid3}, we conclude from \eqref{disev1} that
\begin{align}
\bar{a}_0(\bm{e}_{0}^{\bm{u},\ast}, \bm{v})&=\mathcal{R}^{\bm{u}}(\bm{v})=0\quad \forall \bm{v}\in\bm{V}_{h,\Lambda}, \label{resest11}\\
\bar{d}_0({e}_{0}^{\tilde{p},\ast}, \tilde{q})&=\mathcal{R}^{\tilde{p}}(\tilde{q})=0\quad \,  \forall \tilde{q}\in{W}_{h,\Lambda}. \label{resest17}
\end{align} 
Now, by (\ref{resest11}) and by combining \eqref{disres11} and \eqref{disres11_cp} again, it follows that 
\be \label{first_bit}
      |\bm{e}_{0}^{\bm{u},\ast}|_{\bar{a}_0}^{2} \; \reff{err:decomp:1}{=} \;
      \bar{a}_0(\bm{e}_{0}^{\bm{u},\ast},\bm{r}_{h,\Lambda}^{\bm{u}}+ \bm{r}_{2}^{\bm{u}}) =
      \bar{a}_0(\bm{e}_{0}^{\bm{u},\ast} ,\bm{r}_{2}^{\bm{u}}) =
      \bar{a}_0(\bm{e}_{2}^{\bm{u}} ,\bm{r}_{2}^{\bm{u}}) \le \eta_1 \, |\bm{r}_{2}^{\bm{u}}|_{\bar{a}_0}.
\ee
Similarly, we have
\begin{align}
\label{first_p} |{e}_{0}^{\tilde{p},\ast}|_{\bar{d}_0}^{2} &\le \eta_3 \, |r_{2}^{\tilde{p}}|_{\bar{d}_0}.
\end{align}
Using the orthogonality of functions in $S_{\Lambda}$ and $S_{\mathfrak{Q}}$ with respect to $\pi$ 
{once more}, and
the strengthened Cauchy-Schwarz inequality in~\eqref{scbic11a}, it follows that
\begin{align} \nonumber
   |\bm{e}_{0}^{\bm{u},\ast}|_{\bar{a}_0}^2 & \reff{err:decomp:1}{=} 
   |\bm{r}_{h,\Lambda}^{\bm{u}}|_{\bar{a}_0}^2 + |{\bm{r}}_{2}^{\bm{u}}|_{\bar{a}_0}^2 +
   2\bar{a}_0(\bm{r}_{h,\Lambda}^{\bm{u}}, \bm{r}_{2}^{\bm{u}}) =
   |\bm{r}_{h,\Lambda}^{\bm{u}}|_{\bar{a}_0}^2 + |{\bm{r}}_{2}^{\bm{u}}|_{\bar{a}_0}^2 +
   2\bar{a}_0(\bm{r}_{h,\Lambda}^{\bm{u}},\tilde{\bm{r}}_{h,\Lambda}^{\bm{u}}),\nonumber\\
   &\ \ \,\ge
   |\bm{r}_{h,\Lambda}^{\bm{u}}|_{\bar{a}_0}^2 + |{\bm{r}}_{2}^{\bm{u}}|_{\bar{a}_0}^2 -
   2 \gamma_{1} \, |\bm{r}_{h,\Lambda}^{\bm{u}}|_{\bar{a}_0} |\tilde{\bm{r}}_{h,\Lambda}^{\bm{u}}|_{\bar{a}_0} \nonumber \\
   &\ \ \,\ge
   |\bm{r}_{h,\Lambda}^{\bm{u}}|_{\bar{a}_0}^2 + |{\bm{r}}_{2}^{\bm{u}}|_{\bar{a}_0}^2 -
   2 \gamma_{1} \, |\bm{r}_{h,\Lambda}^{\bm{u}}|_{\bar{a}_0} |{\bm{r}}_{2}^{\bm{u}}|_{\bar{a}_0} \ge
   (1-\gamma_{1}^2) |\bm{r}_{2}^{\bm{u}}|_{\bar{a}_0}^2.
   \label{second_bit}
\end{align}
In  the same manner, using \eqref{err:decomp:2} and \eqref{scbic11b}, we can show that
\begin{align}
\label{second_p}|e_{0}^{\tilde{p},\ast}|_{\bar{d}_0}^2&\ge (1-\gamma_{2}^2) |r_{2}^{\tilde{p}}|_{\bar{d}_0}^2.
\end{align}
Combining \eqref{first_bit} with \eqref{second_bit} and \eqref{first_p} with \eqref{second_p} gives the upper bounds in~\eqref{eq:mainlam11}.
\end{proof}

\smallskip

Combining {the} individual error estimates in~\eqref{err:estimates} we now define the total error estimate
\begin{align}\label{mainest11}
\eta:=(\eta_1^2+\eta_2^2+\eta_3^2)^{1/2}.
\end{align}
The main result below establishes an equivalence between $\eta$ and the SG-MFEM error.

\smallskip

\begin{theorem}\label{mainth3}
Let $(e^{\bm{u}},e^{p},e^{\tilde{p}}) \in \bm{\mathcal V} \times \mathcal {W} \times {\mathcal W}$ be the error in the SG-MFEM approximation $(\bm{u}_{h,\Lambda},p_{h,\Lambda},\tilde{p}_{h,\Lambda}) \in \bm{V}_{h,\Lambda} \times W_{h, \Lambda} \times W_{h,\Lambda}$ of the solution to (\ref{scm12}). Then
\begin{align}
C_6 \, \eta \, \le|||(e^{\bm{u}},e^{p},e^{\tilde{p}})||| \le \frac{C_7}{\sqrt{1-\gamma^2}\sqrt{1-\Theta^2}}\, \eta,
\end{align}
where $C_6$ and $C_7$ are defined in Theorem \ref{mainth1},
$\gamma := \max\{\gamma_1,\gamma_2\} \in [0, 1)$ with the CBS constants
$\gamma_1$ and $\gamma_2$ from \eqref{scbic11},
and $\Theta \in [0,1)$ is the constant given in Lemma \ref{mainth2}.
\end{theorem}
 \begin{proof}
Combining the results of Theorem \ref{mainth1}, Lemma \ref{mainth2} and Lemma \ref{mainlam11} leads to the stated result.
\end{proof}

\smallskip

\begin{corollary} \label{corr_E}
Let $(e^{\bm{u}},e^{p},e^{\tilde{p}}) \in \bm{\mathcal V} \times \mathcal {W} \times {\mathcal W}$ be the error in the SG-MFEM approximation $(\bm{u}_{h,\Lambda},p_{h,\Lambda},\tilde{p}_{h,\Lambda}) \in \bm{V}_{h,\Lambda} \times W_{h, \Lambda} \times W_{h,\Lambda}$ of the solution to \eqref{scm12}. Then 
\begin{align}
{\cal E} \le \frac{C_7}{\sqrt{1-\gamma^2}\sqrt{1-\Theta^2}} \eta, 
\end{align}
where 
${\cal E} \,{:=}\,
 \sqrt{\alpha||\mathbb{E}(\nabla\bm{e}^{\bm{u}})||_{L^{2}(D)}^2 \,{+}\,
 (\alpha^{-1}+(\alpha\beta)^{-1})||\mathbb{E}(e^p)||_{L^{2}(D)}^2 \,{+}\,
 (\alpha\beta)^{-1} ||\mathbb{E}(e^{\tilde{p}})||_{L^{2}(D)}^2}
$
defines an alternative norm of the error and $C_7$, $\gamma$ and $\Theta$ are the same constants as in Theorem \ref{mainth3}.
\end{corollary}

\begin{proof}
The proof directly follows from Theorem \ref{mainth3}, using Jensen's inequality and the definition of the norm $||| \cdot |||$ {in \eqref{normden}}.
\end{proof}

\section{Proxies for the potential error reduction in an adaptive setting} \label{sec5}
%

%
Recall that the \emph{spatial} error estimators 
$\tilde{\bm{e}}^{\bm{u}}_{h,\Lambda} \in \widetilde{\bm{V}}_{h,\Lambda}$  and 
$\tilde{e}_{h,\Lambda}^{\tilde{p}} \in \widetilde{W}_{h,\Lambda}$ {contributing to $\eta_{1}$ and $\eta_{3}$} 
satisfy \eqref{cdisressp11} and \eqref{cdisressp13}, respectively, and 
\rbl{that} the \emph{parametric} error estimators
$\bm{e}^{\bm{u}}_{h,\mathfrak{Q}} \in {\bm{V}}_{h,\mathfrak{Q}}$ and 
$e^{\tilde{p}}_{h,\mathfrak{Q}} \in {W}_{h,\mathfrak{Q}}$ satisfy \eqref{disrespar11} and \eqref{disrespar13}. 
\rbl{Let us also} define a {\emph{third}} spatial error estimator
$\tilde{e}^{p}_{h,\Lambda} \in \widetilde{W}_{h,\Lambda}$ that satisfies
\begin{align}\label{errredeq12}
\bar{c}(\tilde{e}^{{p}}_{h,\Lambda},q)=\mathcal{R}^{p}(q)\quad \forall q\in \widetilde{W}_{h,\Lambda}.
\end{align}
Combining the three spatial error estimators and the two parametric error estimators
\rbl{gives}
\bse \label{err_red_est}
   \begin{align}
        \eta_{h^\ast,\Lambda} &:=
        \big(|\tilde{\bm{e}}_{h,\Lambda}^{\bm{u}}|_{\bar{a}_0}^2+|\tilde{e}^{p}_{h,\Lambda}|_{\bar{c}}^{2} +
        | \tilde{e}_{{h,\Lambda}}^{\tilde{p}}|_{\bar{d}_0}^2\big)^{1/2},
        \label{err_red_est_x}
        \\
        \eta_{h,\mathfrak{Q}} &:=
        \big(|\bm{e}_{h,\mathfrak{Q}}^{\bm{u}}|_{\bar{a}_0}^2 +
        |{e}_{h,\mathfrak{Q}}^{\tilde{p}}|_{\bar{d}_0}^2\big)^{1/2}.
        \label{err_red_est_y}
   \end{align}
\ese
\rbl{We will use these as error reduction proxies within an adaptive refinement scheme.} 


To simplify notation, let ${\cal U} := (\bm{u}, p, \tilde{p}) \in \bm{\mathcal{V}} \times {\cal W} \times{\cal {W}}$ denote the exact solution to~\eqref{scm12} and let ${\cal U}_{h,\Lambda} := (\bm{u}_{h,\Lambda},p_{h,\Lambda},\tilde{p}_{h,\Lambda}) \in  \bm{V}_{h,\Lambda}\times W_{h,\Lambda}\times W_{h,\Lambda}$ denote the SG-MFEM approximation. Similarly, let ${\cal U}_{h^\ast,\Lambda} := (\bm{u}_{h^\ast,\Lambda},p_{h^\ast,\Lambda},\tilde{p}_{h^\ast,\Lambda}) \in 
  \bm{V}_{h^\ast,\Lambda}\times W_{h^\ast,\Lambda}\times W_{h^\ast,\Lambda}$ (resp., ${\cal U}_{h,\Lambda^\ast} := (\bm{u}_{h,\Lambda^\ast},p_{h,\Lambda^\ast},\tilde{p}_{h,\Lambda^\ast}) \in \bm{V}_{h,\Lambda^\ast}\times W_{h,\Lambda^\ast}\times W_{h,\Lambda^\ast}$) denote the enhanced SG-MFEM approximation corresponding to the pair $\bm{V}_h^\ast$--$W_h^\ast$ of enriched finite element spaces (resp., the enriched polynomial space $S_\Lambda^\ast$). That is, ${\cal U}_{h^\ast,\Lambda}$ (resp., ${\cal U}_{h,\Lambda^\ast}$) represents the Galerkin solution that would be obtained by enriching the finite element spaces associated with the spatial domain (resp., the polynomial space associated with the parameter domain). 
\rbl{From the triangle inequality we have
\[
   |||{\cal U} - {\cal U}_{h,\Lambda}||| -    |||{\cal U} - {\cal U}_{h^\ast,\Lambda}|||  \le 
   |||{\cal U}_{h^\ast,\Lambda}-{\cal U}_{h,\Lambda}|||.
\]
\rbl{Thus, we can quantify the potential reduction in the $|||\cdot|||$-norm of the error that would be achieved by 
enriching the finite element spaces, by  estimating the quantity 
$|||{\cal U}_{h^\ast,\Lambda}-{\cal U}_{h,\Lambda}|||$}. Similarly, we can {quantify} the \rbl{potential reduction}
 that would be achieved by enriching the polynomial space for the parametric part, by estimating $|||{\cal U}_{h,\Lambda^\ast}-{\cal U}_{h,\Lambda}|||$}.  The next result shows that $\eta_{h^\ast,\Lambda}$ and $\eta_{h,\mathfrak{Q}}$ provide reliable and efficient proxies for $|||{\cal U}_{h^\ast,\Lambda}-{\cal U}_{h,\Lambda}|||$ and $|||{\cal U}_{h,\Lambda^\ast}-{\cal U}_{h,\Lambda}|||$, respectively. 

\begin{theorem}\label{mainth4}
{Let ${\cal U}_{h,\Lambda} \in \bm{V}_{h,\Lambda}\times W_{h,\Lambda}\times W_{h,\Lambda}$ be the SG-MFEM approximation of the solution to~\eqref{scm12} and define the enhanced SG-MFEM approximations ${\cal U}_{h^\ast,\Lambda}  \in \bm{V}_{h^\ast,\Lambda}\times W_{h^\ast,\Lambda}\times W_{h^\ast,\Lambda}
$ and ${\cal U}_{h,\Lambda^\ast} \in \bm{V}_{h,\Lambda^\ast}\times W_{h,\Lambda^\ast}\times W_{h,\Lambda^\ast} $ as above.} Then,
\begin{align}
   C_6 \,  \eta_{h^\ast,\Lambda} & \le
   |||{\cal U}_{h^\ast,\Lambda}-{\cal U}_{h,\Lambda}||| 
   \le
   \frac{{\widehat{C}_7}}{\sqrt{1-\gamma^2}} \, \eta_{h^\ast,\Lambda},
   \label{errredin1}
\\
   C_6 \, \eta_{h,\mathfrak{Q}} & \le
   |||{\cal U}_{h,\Lambda^\ast}-{\cal U}_{h,\Lambda}|||
   \le
   {\widehat{C}_7} \, \eta_{h,\mathfrak{Q}},
   \label{errredin2}
\end{align}
{where $\widehat{C}_7 := \frac{\widehat{C}_2}{\widehat{C}_1 E_{\min}}$ 
with $\widehat{C}_1$ and $\widehat{C}_2$ defined as in Lemma~\ref{diswelpose11} 
\rbl{except that the inf--sup constant is} associated with the finite element spaces $\bm{V}_{h}^{\ast}$ and $W_{h}^{\ast}$}. \rbl{The constant}  $C_6$ is defined as in Theorem~\ref{mainth1}
and $\gamma \in [0, 1)$ is \rbl{the constant}  in Theorem~\ref{mainth3}.
\end{theorem}

\begin{proof}
Let us prove~\eqref{errredin1},
the proof of~\eqref{errredin2} proceeds in the same way.
The enhanced approximation 
${\cal U}_{h^\ast,\Lambda} = (\bm{u}_{h^\ast,\Lambda},p_{h^\ast,\Lambda},\tilde{p}_{h^\ast,\Lambda})$
satisfies
\begin{subequations} \label{r1disver11enrich}
    \begin{align}
         a({\bm{u}_{h^\ast,\Lambda}},\bm{v})+b(\bm{v},{p_{h^\ast,\Lambda}})&=f(\bm{v})
         \quad \, \, \forall \bm{v}\in \bm{V}_{h^\ast,\Lambda},
         \label{r1cdisressp11}
    \\
         b({\bm{u}_{h^\ast,\Lambda}},q)-c({\tilde{p}_{h^\ast, \Lambda}},q)&=0 
         \quad \qquad \forall q\in  W_{h^\ast,\Lambda},
         \label{r1cdisressp12}
    \\
         -c({p_{h^\ast,\Lambda}},r)+d({\tilde{p}_{h^\ast, \Lambda}},r)&=0 
         \quad \qquad \forall r \in  W_{h^\ast,\Lambda}.
         \label{r1cdisressp13}
    \end{align}
\end{subequations}
Recalling the definitions of the residual functionals in~\eqref{resid} and
combining~\eqref{cdisressp}, (\ref{errredeq12}) and (\ref{r1disver11enrich}) implies
\begin{subequations} \label{r2disver11enrich}
\begin{align}\label{r2cdisressp11}
\bar{a}_0(\tilde{\bm{e}}_{h,\Lambda}^{\bm{u}},\bm{v})&=a(\bm{u}_{h^\ast,\Lambda}-\bm{u}_{h,\Lambda},\bm{v})+b(\bm{v},p_{h^\ast,\Lambda}-p_{h,\Lambda}) \quad\forall \bm{v}\in \widetilde{\bm{V}}_{h,\Lambda},\\
\bar{c}(\tilde{e}_{{h,\Lambda}}^p,q)&=b(\bm{u}_{h^\ast,\Lambda}-\bm{u}_{h,\Lambda},q)-c(\tilde{p}_{h^\ast,\Lambda}-\tilde{p}_{h,\Lambda},q) \,  \, \quad\forall q\in \widetilde{{W}}_{h,\Lambda},\label{r2cdisressp12}\\
\bar{d}_0(\tilde{e}_{h,\Lambda}^{\tilde{p}},r)&=-c(p_{h^\ast,\Lambda}-p_{h,\Lambda},r)+d(\tilde{p}_{h^\ast,\Lambda}-\tilde{p}_{h,\Lambda},r) \quad\forall r\in \widetilde{{W}}_{h,\Lambda}.\label{r2cdisressp13}
\end{align}
\end{subequations}
Substituting $\bm{v} \,{=}\, \tilde{\bm{e}}_{h,\Lambda}^{\bm{u}}$, $q \,{=}\, \tilde{e}_{h,\Lambda}^p$ and $r \,{=}\, \tilde{e}_{h,\Lambda}^{\tilde{p}}$ into (\ref{r2disver11enrich}) and using Lemma~\ref{tecre11}  leads to
\begin{subequations} \label{r3disver11enrich}
   \begin{align}
      |\tilde{\bm{e}}_{h,\Lambda}^{\bm{u}}|_{\bar{a}_0}^2 & \le
      \left(
              E_{\max} \, |\bm{u}_{h^\ast,\Lambda} - \bm{u}_{h,\Lambda}|_{\bar{a}_0} +
              \sqrt{d}\, |p_{h^\ast,\Lambda} - p_{h,\Lambda}|_{\bar{c}}
      \right)
      |\tilde{\bm{e}}_{h,\Lambda}^{\bm{u}}|_{\bar{a}_0},
      \label{r3cdisressp11}
   \\
      |\tilde{e}_{{h,\Lambda}}^p|_{\bar{c}}^2 & \le
      \left( \sqrt{d} \, |\bm{u}_{h^\ast,\Lambda} - \bm{u}_{h,\Lambda}|_{\bar{a}_0} +
      |\tilde{p}_{h^\ast,\Lambda}-\tilde{p}_{h,\Lambda}|_{\bar{d}_0} \right)
      |\tilde{e}_{h,\Lambda}^p|_{\bar{c}},
      \label{r3cdisressp12}
   \\
      |\tilde{e}_{h,\Lambda}^{\tilde{p}}|_{\bar{d}_0}^2 & \le
      \Big( |p_{h^\ast,\Lambda} - p_{h,\Lambda}|_{\bar{c}} +
      E_{\max} \, |\tilde{p}_{h^\ast,\Lambda} - \tilde{p}_{h,\Lambda}|_{\bar{d}_0} \Big)\,
      |\tilde{e}_{h,\Lambda}^{\tilde{p}}|_{\bar{d}_0}.
      \label{r3cdisressp13}
\end{align}
\end{subequations}
Combining all three estimates in (\ref{r3disver11enrich}) leads to the left-hand inequality in  (\ref{errredin1}). 


In order to prove the right-hand inequality in (\ref{errredin1}), 
{we need to use a discrete stability result which is analogous to the one given}
in Lemma~\ref{diswelpose11}. Since 
${\cal U}_{h^\ast,\Lambda}-{\cal U}_{h,\Lambda} \in
  \bm{V}_{h^\ast,\Lambda}\times {W}_{h^{\ast},\Lambda}\times {W}_{h^{\ast},\Lambda}$,
there exists
${\cal V} := (\bm{v},q,r) \in  \bm{V}_{h^\ast,\Lambda}\times {W}_{h^{\ast},\Lambda}\times {W}_{h^{\ast},\Lambda}$
with
\be \label{errred_est0_ab}
      |||{\cal V}||| {=} \big( |\bm{v}|_{\bar{a}_0}^2 + |q|_{\bar{c}}^2 + |r|_{\bar{d}_0}^2 \big)^{1/2}
      \le
      {\widehat{C}_2}  \, |||{\cal U}_{h^\ast,\Lambda}-{\cal U}_{h,\Lambda}|||
\ee
satisfying 
\begin{equation} \label{errred_est1_ab}
   {E}_{\min}\, \widehat{C}_1 \, ||| {\cal U}_{h^\ast,\Lambda}-{\cal U}_{h,\Lambda} |||^2 \le
   \mathcal{B}({\cal U}_{h^\ast,\Lambda}-{\cal U}_{h,\Lambda}; {\cal V}),
\end{equation}
{where $\widehat{C}_{1}$ and $\widehat{C}_{2}$ are defined analogously to the constants $C_{1}^{\ast}$ and $C_{2}^{\ast}$ in Lemma~\ref{diswelpose11} but with $\gamma_{h}$ replaced by the inf--sup constant $\gamma_{h}^{\ast}$ associated with the spaces $\bm{V}_{h}^{\ast}$ and $W_{h}^{\ast}$}. Recalling the definitions of $\bm{V}_{h^\ast,\Lambda}$ and ${W}_{h^{\ast},\Lambda}$ and the associated decompositions in~\eqref{enriched_spacesV} and~\eqref{enriched_spacesW},
we may decompose
${\cal V} \in \bm{V}_{h^\ast,\Lambda}\times {W}_{h^{\ast},\Lambda}\times {W}_{h^{\ast},\Lambda}$ as
${\cal V} = {\cal V}_{h,\Lambda}+\widetilde{\cal V}_{h,\Lambda}$, where
${\cal V}_{h,\Lambda} := (\bm{v}_{h,\Lambda}, q_{h,\Lambda}, r_{h,\Lambda}) \in
  \bm{V}_{h,\Lambda} \times {W}_{h,\Lambda} \times {W}_{h,\Lambda}$
and
$\widetilde{\cal V}_{h,\Lambda} := (\tilde{\bm{v}}_{h,\Lambda}, \tilde{q}_{h,\Lambda}, \tilde{r}_{h,\Lambda}) \in
  \widetilde{\bm{V}}_{h,\Lambda} \times \widetilde{W}_{h,\Lambda} \times \widetilde{W}_{h,\Lambda}$.
Since $\mathcal{B}({\cal U}_{h^\ast,\Lambda}-{\cal U}_{h,\Lambda}; {\cal V}_{h,\Lambda}) = 0$,
we use equations~\eqref{r2disver11enrich} and the Cauchy-Schwarz inequality to conclude from~\eqref{errred_est1_ab} that
%
%
%
\begin{align} \label{r1stab2}
   {E}_{\min}\, {\widehat{C}_1} \, ||| {\cal U}_{h^\ast,\Lambda}-{\cal U}_{h,\Lambda} |||^2 & \le
   \mathcal{B}({\cal U}_{h^\ast,\Lambda}-{\cal U}_{h,\Lambda}; \widetilde{\cal V}_{h,\Lambda})
   \nonumber
   \\
   & \!\!\!\reff{r2disver11enrich}{=}
   \bar{a}_0(\tilde{\bm{e}}_{h,\Lambda}^{\bm{u}}, \tilde{\bm{v}}_{h,\Lambda}) +
   \bar{c}(\tilde{e}_{h,\Lambda}^p,\tilde{q}_{h,\Lambda}) +
   \bar{d}_0(\tilde{e}_{h,\Lambda}^{\tilde{p}},\tilde{r}_{h,\Lambda})\qquad\qquad\ \ 
   \nonumber
   \\
   & {\le}\; |\tilde{\bm{e}}_{h,\Lambda}^{\bm{u}}|_{\bar{a}_0} |\tilde{\bm{v}}_{h,\Lambda}|_{\bar{a}_0}  \,{+}\,
   |\tilde{e}_{h,\Lambda}^p|_{\bar{c}} |\tilde{q}_{h,\Lambda}|_{\bar{c}}  \,{+}\,
   |\tilde{e}_{h,\Lambda}^{\tilde{p}}|_{\bar{d}_0} |\tilde{r}_{h,\Lambda}|_{\bar{d}_0}.
\end{align}
We will now estimate $|\tilde{\bm{v}}_{h,\Lambda}|_{\bar{a}_0}$, $|\tilde{q}_{h,\Lambda}|_{\bar{c}}$,
and $|\tilde{r}_{h,\Lambda}|_{\bar{d}_0}$.
Using the strengthened Cauchy-Schwarz inequality in~\eqref{scbic11a}, we have
   \begin{align*}
      |\tilde{\bm{v}}_{h,\Lambda}|_{\bar{a}_0}^2 & =
      |\bm{v}|_{\bar{a}_0}^2 - 2\bar{a}_0(\bm{v}_{h,\Lambda},\tilde{\bm{v}}_{h,\Lambda}) -
      |\bm{v}_{h,\Lambda}|_{\bar{a}_0}^2 \le
      |\bm{v}|_{\bar{a}_0}^2 + 2 \gamma_1 |\bm{v}_{h,\Lambda}|_{\bar{a}_0} |\tilde{\bm{v}}_{h,\Lambda}|_{\bar{a}_0} -
      |\bm{v}_{h,\Lambda}|_{\bar{a}_0}^2
      \\[3pt]
      & \le
      |\bm{v}|_{\bar{a}_0}^2 + |\bm{v}_{h,\Lambda}|^2_{\bar{a}_0} + \gamma_1^2 |\tilde{\bm{v}}_{h,\Lambda}|^2_{\bar{a}_0} -
      |\bm{v}_{h,\Lambda}|_{\bar{a}_0}^2 =
      |\bm{v}|_{\bar{a}_0}^2 + \gamma_1^2 |\tilde{\bm{v}}_{h,\Lambda}|^2_{\bar{a}_0},
\end{align*}
which yields
\begin{subequations}\label{r1bnd1}
\begin{equation}
   |\tilde{\bm{v}}_{h,\Lambda}|_{\bar{a}_0} \le (1-\gamma_1^2)^{-1/2} |\bm{v}|_{\bar{a}_0}.
\end{equation}
In the same way, using~\eqref{scbic11b}, we obtain
\begin{equation}
   |\tilde{{q}}_{h,\Lambda}|_{\bar{c}}  \le (1-\gamma_2^2)^{-1/2} |q|_{\bar{c}},\qquad
   |\tilde{r}_{h,\Lambda}|_{\bar{d}_0}  \le (1-\gamma_2^2)^{-1/2} |r|_{\bar{d}_0}.
\end{equation}
\end{subequations}
%
%
Combining (\ref{r1stab2}) and (\ref{r1bnd1}), using (\ref{errred_est0_ab}), and recalling
the definition of $\eta_{h^\ast,\Lambda}$ in~(\ref{err_red_est_x}) we finally prove the right-hand inequality in~(\ref{errredin1})
with $\gamma = \max\{\gamma_1,\,\gamma_2\}$. \end{proof}

%


\subsection{The parametric error estimators}
\label{sec:param:estimators}

In this subsection, we discuss some important properties of the \emph{parametric} error estimators {$\bm{e}_{h,\mathfrak{Q}}^{\bm{u}}$ and $e_{h,\mathfrak{Q}}^{\tilde{p}}$. Recall that these contribute to $\eta_{1}$ and $\eta_{3}$ and hence the total error estimate $\eta$ as well as the error reduction proxy $\eta_{h,\mathfrak{Q}}$.}

Let $\mathfrak{Q}$ be any finite detail index set, i.e., let
$\mathfrak{Q} = \{\bm{\alpha} \in \mathfrak{I};\; \bm{\alpha} \notin \Lambda\}$.
%
Since the subspaces $\bm{V}_h \otimes \textrm{span}\{\psi_{\bm{\alpha}}\}$ with $\bm{\alpha} \in \mathfrak{Q}$
are pairwise orthogonal with respect to the inner product $\bar{a}_0(\cdot,\cdot)$, the parametric error estimator
$\euh{\GQ}$ defined in \eqref{disrespar11} can be decomposed into {separate} contributions associated with {the} individual multi-indices $\balpha \in \GQ$ as follows:
\be \label{param:estimator:sum}
      \euh{\GQ} = \sum\limits_{\balpha \in \GQ} \euh{\balpha}\quad
      \hbox{with}\quad
      |\euh{\GQ}|_{\bar{a}_0}^2 = \sum\limits_{\balpha \in \GQ} |\euh{\balpha}|_{\bar{a}_0}^2,
\ee
where, for each $\balpha \in \GQ$, the estimator
$\euh{\balpha} \in \bm{V}_h \otimes \textrm{span}\{\psi_{\bm{\alpha}}\}$ satisfies
\be \label{param:estimator:def}
      \bar{a}_0(\euh{\balpha}, \bm{v}_h \psi_{\balpha}) = \mathcal{R}^{\bm{u}}(\bm{v}_h \psi_{\balpha})\quad
      \forall\, \bm{v}_h \in \bm{V}_h.
\ee
{Hence to compute $|\euh{\GQ}|_{\bar{a}_0}^2$ we simply solve a set of decoupled Poisson problems associated with the finite element space $\bm{V}_{h}$. A similar decomposition holds for the {parametric} error estimator $\etph{\GQ}$ defined in \eqref{disrespar13}, and we will denote the individual estimators associated with each multi-index $\balpha\,{\in}\,\GQ$ by $\etph{\balpha} \,{\in}\, W_h \,{\otimes}\, \textrm{span}\{\psi_{\bm{\alpha}}\}$.} 

{Thanks to Theorem~\ref{mainth4}, we see that} the quantity $\eta_{h,\balpha} := \big(|\euh{\balpha}|_{\bar{a}_0}^2 + |\etph{\balpha}|_{\bar{d}_0}^2\big)^{1/2}$
(cf.~\eqref{err_red_est_y}) {can be used to estimate} the error reduction that would be achieved
by adding only \emph{one} multi-index $\balpha \in \GI \setminus \Lambda$
to the current index set $\Lambda$ and computing the corresponding enhanced approximation
${\cal U}_{h,\Lambda^\ast}$ with $\Lambda^\ast = \Lambda \cup {\{\balpha\}}$.

An important aspect of our error estimation strategy (and a key ingredient of the adaptive algorithm presented below) is the choice of the detail index set $\GQ$. Let $\bm{t}^{(n)} = (t^{(n)}_1,t^{(n)}_2,\ldots) \in\mathfrak{I}$ be the Kronecker delta index
for the coordinate $n\in \mathbb{N}$, i.e.,  $t^{(n)}_j= \delta_{jn}$ for any $j\in \mathbb{N}$.
Next, for any finite index set $\Lambda$, we define $\Lambda_{\infty}^\ast$
as the infinite index set given by $\Lambda_{\infty}^\ast=\Lambda\cup\mathfrak{Q}_{\infty}$, where
\be \label{Qinfty}
      \mathfrak{Q}_{\infty} :=
      \{\bm{\alpha}\in\mathfrak{I}\setminus\Lambda; \;
      \bm{\alpha}=\bm{\tau} \pm \bm{t}^{(n)}\;\ \forall\,\bm{\tau}\in\Lambda,\ \forall\, n=1,2,\ldots\}
\ee
denotes the boundary of $\Lambda$.
The following result follows from Co\-rol\-la\-ry~4.3 in \cite{bs16}.


\smallskip

\begin{lemma} \label{parestlem2}
Let the detail index set $\mathfrak{Q}$ be a finite subset of the index set
$\mathfrak{I}\setminus \Lambda^\ast_{\infty}$. 
Then the parametric error estimators $\euh{\GQ}$ and $\etph{\GQ}$ are identically equal to zero.
\end{lemma}


{Hence, nonzero contributions to $\euh{\GQ}$ and $\etph{\GQ}$ are associated only with the indices from the boundary index set $\mathfrak{Q}_{\infty}$.
This has two consequences. First, for the error estimation to be effective, $\GQ$ should be chosen as a sufficiently large (finite) subset of $\GQ_\infty$. Second, for the adaptive algorithm to be efficient, the index set $\Lambda$ should only be enriched at each step with multi-indices from the index set $\GQ_\infty$.}

\subsection{\rbl{A rudimentary} adaptive algorithm} \label{sec:adapt:alg}
\rbl{As is conventional, our adaptive approach is to  solve the SG-MFEM problem~\eqref{disver11} and estimate the energy error by computing $\eta$ in~\eqref{mainest11}. The approximation is then refined (by adaptively enriching the underlying approximation spaces) until either $\eta < \hbox{\tt tol}$, 
where {\tt tol} is a user-prescribed tolerance, or else the total number of degrees
of freedom in the discrete problem exceeds some specified maximum value}.
Let us first focus on the computation of $\eta = (\eta_1 + \eta_2 +\eta_3)^{1/2}$, 
where $\eta_1,\ \eta_2,\ \eta_3$ are defined in~(\ref{err:estimates}).
Recall that $\eta_2$ is calculated directly using~\eqref{e_0_p_def}.
To compute the spatial and parametric error estimators that contribute to $\eta_1$, $\eta_3$
(see~\eqref{eta1_Def},~\eqref{eta3_Def}), one needs to specify
the detail finite element spaces $\widetilde{\bm{V}}_h$, $\widetilde{W}_h$
and the detail polynomial space $S_\GQ$ on the parameter domain~$\Gamma$.
In our algorithm, $\widetilde{\bm{V}}_h$ and $\widetilde{W}_h$ will span local bubble functions,
for which we have two alternatives: {(\rbl{option I})} piecewise polynomials of the same order as $\bm{V}_h$ and $W_h$, respectively on a uniformly refined mesh~${\mathcal T}_{h/2}$; and {(option II)} piecewise polynomials of a higher order on the mesh~${\mathcal T}_{h}$. In {both cases}, {the computation of the spatial estimators} $\tilde{\bm{e}}_{h,\Lambda}^{\bm{u}}$ and $\tilde{e}_{h,\Lambda}^{\tilde{p}}$ {is broken down over the elements} $K \in {\mathcal T}_h$ using
the standard element residual technique (see, e.g.,~\cite{MJ}).
\rbl{One should view the approach that we have adopted  as a  proof of concept.
Greater efficiency could almost certainly be achieved using  a more sophisticated refinement strategy, 
such as the multilevel strategy  developed in~\cite{cpb18}, where
distinct solution modes are associated with a finite element space on a different mesh.}

The construction of the detail index set $\GQ$ {(that defines  $S_\GQ$)} is motivated by Lemma~\ref{parestlem2}.
Specifically, we will use the following finite subset of the index set $\GQ_{\infty}$ defined in~\eqref{Qinfty}:
\be \label{indsetQ}
   \mathfrak{Q} = \left\{ 
                           \balpha \in \mathfrak{I} \setminus \Lambda;\;
                           \balpha= \bm{\tau} \pm \bm{t}^{n}\;\
                           \forall \bm{\tau} \in \Lambda,\;\ \forall n = 1,2,\ldots, {M_{\Lambda} +1}
                           \right\},
\ee
where $M_\Lambda \in \mathbb{N}$ {(the number of active parameters)} is defined as 
\begin{equation*}
  M_\Lambda := 
  \begin{cases}
           0        &  \text{if $\Lambda = \{ (0,0,\dots) \} $}, \\
           \max \big\{ \max (\supp (\balpha) ) ;\; \balpha \in \Lambda \setminus \{ (0,0,\dots) \} \big\} & \text{otherwise}.
  \end{cases}
\end{equation*}
%
%
%
%
Given the index set $\GQ$ in~\eqref{indsetQ},  the parametric error estimators ${\bm{e}}_{h,\mathfrak{Q}}^{\bm{u}}$ and ${e}_{h,\mathfrak{Q}}^{\tilde{p}}$
contributing to $\eta_1$, $\eta_3$ are computed from the corresponding individual error estimators
${\bm{e}}_{h,{\balpha}}^{\bm{u}}$ and ${e}_{h,\balpha}^{\tilde{p}}$ for each $\balpha\in\mathfrak{Q}$
as explained in Section~\ref{sec:param:estimators}
(see, e.g.,~\eqref{param:estimator:sum}--\eqref{param:estimator:def}).


Let us now describe the refinement procedure. {If the estimated error $\eta$ is too large then}, in order to compute a more accurate approximation,
one needs to enrich the finite-dimensional subspaces
$\bm{V}_{h,\Lambda} = \bm{V}_{h} \otimes S_{\Lambda}$ and $W_{h,\Lambda} = W_{h} \otimes S_{\Lambda}$.
Recall that the quantities (see~\eqref{err_red_est})
\be \label{err:reductions}
        \eta_{h^\ast,\Lambda} := 
        \big(|\tilde{\bm{e}}_{h,\Lambda}^{\bm{u}}|_{\bar{a}_0}^2 + |\tilde{e}^{p}_{h,\Lambda}|_{\bar{c}}^{2} +
        |\tilde{e}_{{h,\Lambda}}^{\tilde{p}}|_{\bar{d}_0}^2\big)^{1/2}\ \ \hbox{and}\ \
        \eta_{h,\balpha} :=
        \big(|\bm{e}_{h,\balpha}^{\bm{u}}|_{\bar{a}_0}^2 +
        |{e}_{h,\balpha}^{\tilde{p}}|_{\bar{d}_0}^2\big)^{1/2} 
\ee
(for $\balpha \in \GQ$) provide proxies for the potential error reductions associated with spatial and parametric enrichment, respectively
(in the latter case, the enrichment is associated with adding {a single} index $\balpha \in \GQ$); see Theorem~\ref{mainth4}. We use the 
dominant proxy to guide the enrichment of $\bm{V}_{h,\Lambda}$ and $W_{h,\Lambda}$. 
More precisely, if $\eta_{h^\ast,\Lambda} \ge  \tau \,  \eta_{h,\mathfrak{Q}}$ with a refinement weighting factor $\tau \geq 1$, 
then the finite element spaces $\bm{V}_{h}$ and $W_{h}$ are enriched
\rbl{by refining the finite element mesh on $D$};
otherwise, the polynomial space $S_\Lambda$ on $\Gamma$ is enriched by adding at least one new index to the set $\Lambda$. In the latter case, we enrich $\Lambda$ with the index $\balpha \in \GQ$ corresponding to the largest of the proxies $\eta_{h,\balpha}$ as well as any additional indices $\balpha$ for which $\eta_{h,\balpha} \ge  \eta_{h^\ast,\Lambda}$.   


Our adaptive strategy is presented in Algorithm~\ref{algor_1}. Starting with an inf-sup stable pair $\bm{V}_{h_0}\!$--$\, W_{h_0}$ of finite element spaces
on a coarse mesh ${\cal T}_{h_0}$ and with an initial index set $\Lambda_0$
(typically, $\Lambda = \{(0,0,0,\ldots)\}$ or $\Lambda = \{(0,0,0,\ldots),\; (1,0,0,\ldots)\})$, the algorithm
generates two sequences of finite element spaces
\[
   \bm{V}_{h_0} \subseteq \bm{V}_{h_1} 
   \subseteq \ldots \subseteq 
   \bm{V}_{h_K} \subset \bm{H}^1_{E_0}(D)
   \quad
   \hbox{and}\quad
   W_{h_0} \subseteq W_{h_1} 
   \subseteq \ldots \subseteq 
   W_{h_K} \subset L^2(D),
\]
a sequence of index sets $\Lambda_0 \subseteq \Lambda_1 \subseteq \Lambda_2 \ldots \subseteq \Lambda_K \subset \mathfrak{I}$,
as well as a sequence of SG-MFEM approximations
 $(\bm{u}_{k},p_{k},\tilde{p}_k) \in \bm{V}_{h_k,\Lambda_k} \times W_{h_k,\Lambda_k}\times W_{h_k,\Lambda_k}$
and the corresponding error estimates $\eta^{(k)}$ ($k \,{=}\, 0,1,\ldots,K$). The refinement weighting factor is chosen to be $\tau=\sqrt{2}$ and the algorithm is terminated when the estimated error is sufficiently small.

\begin{figure}[!ht] 
\begin{algorithm} \label{algor_1}
{\tt Adaptive$\_\,${SG-MFEM}}$\,\big[{\tt tol},\mathcal{B},f, {h_0, \Lambda_0} \big] \rightarrow 
                                                    \big[(\bm{u}_K,p_K,\tilde{p}_K), \eta^{(K)}\big]$\\
{\tt
{\bf for} ${{k}} = 0,1,2,\ldots$ {\bf do}\\[3pt]
\phantom{+}\quad\;$(\bm{u}_{{{k}}},p_k,\tilde{p}_k)\;{\leftarrow}\;
                  \hbox{Solve}\big[\mathcal{B},f,\bm{V}_{h_{k}},W_{h_{k}},{\Lambda}_{{k}}\big]$\\[3pt]
\phantom{+}\quad\;$\left[\delta_{h}, \eta_{h^\ast,\Lambda}\right]\;{\leftarrow}\;
                     \hbox{Error$\_\,$Estimate$\_\,$1}
                     \big[\mathcal{B},f,\bm{u}_{{{k}}}, p_k, \tilde{p}_k, 
                     \widetilde{\bm{V}}_{h_{k}}, \widetilde{W}_{h_{k}} \big]$\\[3pt]
\phantom{+}\quad\;${\GQ}_{{{k}}} \;{\leftarrow}\;
                  \hbox{Detail$\_\,$Index$\_\,$Set}\big[{\Lambda}_{{{k}}}\big]$\\[3pt]
\phantom{+}\quad\;{\bf for} {$i = 1,2,\ldots,\#({\GQ}_{{k}})$} {\bf do}\\[3pt]
\phantom{+}\qquad\ \ $\eta_{h,i}\;{\leftarrow}\;
                     \hbox{Error$\_\,$Estimate$\_\,$2}\big[\mathcal{B},f,\bm{u}_k, p_k, \tilde{p}_k, 
                     \balpha_i \big]$\\[3pt]
\phantom{+}\quad\;{\bf end}\\[1pt]
\phantom{+}\quad $ \eta_{h,\mathfrak{Q}} :=  \left(\sum_{i=1}^{\#({\GQ}_{{k}})} \eta_{h,i}^2\right)^{1/2} $\\[3pt]  
\phantom{+}\quad\;$\eta^{(k)} := \big(\delta_{h}^2 +   \eta_{h,\mathfrak{Q}}^{2}   \big)^{1/2}$\\[3pt]
\phantom{+}\quad\;{\bf if} $\eta^{(k)} < {\tt tol}$ {{\bf then} $K:=k$, {\bf break}}\\[3pt]
\phantom{+}\quad\;{\bf if} $\eta_{h^\ast,\Lambda} \ge  \sqrt{2} \,   \eta_{h,\mathfrak{Q}} $
                  {\bf then}\\[3pt]
\phantom{+}\quad\ \ \;$\bm{V}_{h_{k+1}} := \rbl{\bm{V}_{h_k^*}},\ W_{h_{k+1}} := \rbl{W_{h_k^*}},\
                                       {\Lambda}_{{{k}}+1} := {\Lambda}_{{{k}}}$\\ [3pt]
\phantom{+}\quad\;{\bf else}
                  $\bm{V}_{h_{k+1}} := \bm{V}_{h_k},\ W_{h_{k+1}} := W_{h_k}$\\ [3pt] 
\phantom{+}\quad  \qquad \, ${\Lambda}_{{{k}}+1} := {\Lambda}_{{k}} \cup
                              \big\{\balpha_{j} \in  {\GQ}_{{k}}; \eta_{h,j} := \max_{i} \eta_{h,i} \big\}  \cup \big\{\balpha_i \in {\GQ}_{{k}};\; \eta_{h,i} \, \ge \, \eta_{h^\ast,\Lambda}\big\}$ \\ [1pt]
{\bf end}}
\end{algorithm}
\end{figure}
\newpage 
\vspace*{-20pt}
The algorithm has  \rbl{four} functional building blocks:
\begin{itemize}
\item
{\tt Solve}$\big[\mathcal{B},f,\bm{V}_h,W_h,\Lambda\big]$: a subroutine that generates the 
SG-MFEM approximation
$({\bm{u}_{h,\Lambda}}, {p_{h,\Lambda}}, {\tilde{p}_{h, \Lambda}})
  \in \bm{V}_{h, \Lambda} \times W_{h,\Lambda} \times W_{h,\Lambda}$ satisfying~\eqref{disver11};
\item
{\tt Detail$\_\,$Index$\_\,$Set}$\big[\Lambda\big]$: a subroutine that generates
the {detail} index set $\mathfrak{Q}$ for the given index set $\Lambda$ (see~\eqref{indsetQ});
\item
{\tt Error$\_\,$Estimate$\_\,$1}
$\big[\mathcal{B},f,\bm{u}_{h,\Lambda}, p_{h,\Lambda}, \tilde{p}_{h,\Lambda},
  \widetilde{\bm{V}}_{h}, \widetilde{W}_{h}\big]$: a subroutine that computes
$\delta_h := 
  ( |\tilde{\bm{e}}_{h,\Lambda}^{\bm{u}}|^2_{\bar a_0} +
  |{e}_{0}^{p}|^2_{\bar c} +
  |\tilde{e}_{h,\Lambda}^{\tilde{p}}|^2_{\bar d_0} )^{1/2},
$
the contribution to $\eta$ associated with spatial enrichment, and the spatial error reduction proxy $\eta_{h^\ast,\Lambda}$ in~\eqref{err:reductions};
\item
{\tt Error$\_\,$Estimate$\_\,$2}
$\big[\mathcal{B},f,\bm{u}_{h,\Lambda}, p_{h,\Lambda}, \tilde{p}_{h,\Lambda},\balpha\big]$: a subroutine that computes the contribution to $\eta$ and parametric error reduction proxy 
associated with a single index $\balpha \in \mathfrak{Q}$ (see~\eqref{err:reductions}).
\end{itemize}

\smallskip

The IFISS software~\cite{ifiss} that we use to test the efficiency of our methodology is limited to two-dimensional spatial approximation. It
provides two alternative choices for the {\it spatial refinement} step in Algorithm~5.1
(that is,  the generation of the spaces $\bm{V}_{h_k^*}$ and $W_{h_k^*}$) that is taken 
 whenever the spatial refinement proxy $\eta_{h^\ast,\Lambda}$ 
dominates the \rbl{parametric  refinement proxy $ \eta_{h,\mathfrak{Q}}$}.
In cases where the solution is
spatially smooth, a natural option is to define  $h_k^*$  by taking a {\it uniform refinement}
of the current grid.  In  the computational experiments discussed later,  this 
refinement option is associated with a rectangular subdivision of  the spatial domain.\footnote{Specifically, 
we always use uniform refinement in combination with  the inf--sup stable approximation pairs
$\bm{V}_{h}$--$\, W_{h}$ that are built into the S-IFISS toolbox~\cite{SIFISS}.} 
On the other hand, when solving  spatially singular problems, it is more natural to 
define  $h_k^*$ by a {\it local refinement strategy} in combination with triangular approximation.
In our TIFISS toolbox implementation~\cite{TIFISS} this is done using a standard
 iterative refinement loop
\begin{align*}
\mbox{Solve}\rightarrow\mbox{Estimate}\rightarrow\mbox{Mark}\rightarrow\mbox{Refine}
\end{align*} 
combined with a  bulk parameter marking procedure with marking parameter $\theta=1/2$.

\noindent A complete description of the strategy we are using can be found in~\cite{bespalovrocchi}. 


\section{Incompressible limit case} \label{sec6}
If $\nu=\frac{1}{2}$, then our three-field formulation (\ref{os3}) reduces to the following two-field formulation
representing the Stokes problem:
find $\bm{u}: D\times\Gamma \to \rrb{\mathbb{R}^{d}}$ and $p: D\times\Gamma \to \mathbb{R}$ such that
\begin{subequations}\label{Stokes}
\begin{align}\label{Stokes1}
 -\nabla\cdot\sigmab(\bm{x},\bm{y}) & =\bm{f}(\bm{x}), \quad \bm{x}\in D, \;\bm{y}\in\Gamma, \\
 \nabla\cdot\bm{u}(\bm{x},\bm{y}) & = 0,\quad \quad \, \, \, \bm{x}\in D, \;\bm{y}\in\Gamma, \label{Stokes2}\\
 \bm{u}(\bm{x},\bm{y})&=\bm{g}(\bm{x}),\quad \bm{x} \in \partial D_D,\;\bm{y}\in\Gamma,\\
 \sigmab(\bm{x},\bm{y})\, \bm{n}&= \mathbf{0},\quad \quad \, \, \, \bm{x} \in {\partial D_N},\;\bm{y}\in\Gamma.
\end{align}      
\end{subequations}
Assuming that $\bm{f}\in ( L^{2}(D))^{{d}}$ and $\bm{g}= \bm{0}$ on $\partial D_D$,
the weak formulation of~\eqref{Stokes} and the associated SG-MFEM formulation follow
from~\eqref{scm11a} and~\eqref{disver11}, respectively,
by formally setting the bilinear forms $c(\cdot,\cdot)$ and $d(\cdot,\cdot)$ to zero and
omitting the third components of the weak and Galerkin solutions.
\rbl{In the incompressible limit}, the error estimate $\eta$ defined in~\eqref{mainest11}~becomes
$
   \eta = (\eta_1^2+\eta_2^2)^{1/2},
$
where $\eta_1$ and $\eta_2$ are given in~(\ref{err:estimates}) (with $\eta_2 = {\alpha}^{1/2} ||  \nabla \cdot \bm{u}_{h,\Lambda}||_{\mathcal{W}}$).
%
%
The following result is an immediate consequence of Theorem~\ref{mainth3}.

\begin{theorem}\label{mainth5}
Let $(\bm{e}^{\bm{u}},e^p)$ be the error in the SG-MFEM approximation of the {weak} 
solution to~\eqref{Stokes}.
Then
\begin{align}
{C_6} \, \eta \le |||(\bm{e}^{\bm{u}},e^p)|||_{S}\le \frac{C_7}{\sqrt{1-\gamma^2}\sqrt{1-\Theta^2}} \eta,
\end{align} 
\rbl{where} 
$|||(\bm{v},q)|||_S^2 := \alpha \|\nabla \bm{v}\|^2_{\bm{\mathcal{W}}}
 + \alpha^{-1}\|q\|_{\mathcal{W}}^2$, $C_6:= \big(\sqrt{2}(E_{\max} + \sqrt{d})\big)^{-1}$
  and $C_7$, $\gamma$ and $\Theta$ are as specified  in
 Theorem~\ref{mainth3}.
\end{theorem}

\section{\rbl{Computational results}} \label{sec7}
In this section, we present two numerical examples to validate our theoretical results. 
In the first experiment, we consider a simple test problem with an exact solution and 
investigate the accuracy of the a posteriori error estimate $\eta$. In the second,
 we consider a problem where the Young's modulus depends on a 
 countably infinite set of parameters and investigate the performance 
 of the proposed adaptive algorithm.


\subsection{Exact solution, Dirichlet boundary \rbl{condition}} 
To define a problem of the form \eqref{os3} with an exact solution, 
we choose the spatial domain $D=(0,1)^2$ and impose a Dirichlet  condition 
on the displacement  $\bm{u}$ on the whole boundary.  
Hence, $\partial D= \partial D_D$ and $\partial D_{N} =\emptyset$. 
The uncertain Young's modulus \rbl{is modelled as  
$E := e_0 +  0.1 y_1$} where $y_1\in[-1,1]$ is the image of a mean zero uniform random variable. 
Hence, $E$ is spatially constant \rbl{and $e_0=1$ is the mean}. The body force
\begin{align}
\bm{f}=\begin{cases} 
f_1= -2\alpha \pi^3 \cos(\pi x_2) \sin(\pi x_2) (2\cos(2\pi x_1)-1),\\
f_2= 2\alpha \pi^3 \cos(\pi x_1) \sin(\pi x_1) (2\cos(2\pi x_2)-1),
\end{cases}
\end{align}
is chosen so that the exact displacement is 
\begin{align}
\bm{u}=\begin{cases} 
u_1= \pi\cos(\pi x_2)\sin(\pi x_2) \sin^2(\pi x_1)/E,\\
u_2= -\pi\cos(\pi x_1)\sin(\pi x_1) \sin^2(\pi x_2)/E,
\end{cases}
\end{align}
and \rbl{the exact pressure is} $p=\tilde{p}=0$.
\rbl{For the spatial discretization  we  use  $\bm{Q}_2$--$P_{-1}$--$P_{-1}$ approximation on  
uniform grids of square elements.}\footnote{The combination $\bm{Q}_2$--$P_{-1}$ is one of the most 
effective inf--sup stable approximation pairs in a two-dimensional uniform refinement setting; 
see~\cite[Sect.\,3.3.1]{HDA}.}   
To compute the SG-MFEM solution, we choose $S_{\Lambda}$ to be the space of 
polynomials of degree less than or equal to $k$ in $y_{1}$ on $\Gamma=[-1,1]$. 
\rbl{To assess the quality of the error estimate we will examine} 
\begin{align}
\mbox{Effectivity index} =\frac{\eta}{{\cal E}}, \nonumber
\end{align}
where ${\cal E}$ is the error defined in Corollary \ref{corr_E}, as we vary the SG-MFEM discretisation parameters $h$ and $k$. \rbl{Results  for five representative values of the 
Poisson ratio $\nu$ are} presented in Tables \ref{small_sigma_p3} and \ref{small_sigma_p31}. 
To compute the error estimate $\eta$, we consider two types of finite element detail spaces 
$\widetilde{\bm{V}}_{h}$ $\widetilde{W}_{h}$ based on local bubble functions (option I, option II),
 as explained in Section~\ref{sec:adapt:alg}. Since we have only one parameter $y_1$, 
 the polynomial space $S_{\mathfrak{Q}}$ is chosen to be the set of polynomials of degree equal to $k+1$. The results confirm that our a posteriori error estimate 
 is robust with respect to the Poisson ratio (in the incompressible limit), the finite element 
 mesh size $h$ and the polynomial degree $k$ associated with the parametric approximation.

\begin{table}[ht!]
 \caption[]{{Test problem 1: Effectivity indices for fixed  polynomial degree $k=3.$}  }
    \label{small_sigma_p3}
\begin{center}
    \begin{tabular}{ | l | l | l | l | l | l | l | l | l | l | l | }
      \hline   
    \multicolumn{1}{| c |}{$h$}&{$\nu=.4$ }& $\nu=.49$&$\nu=.499$&$\nu=.4999$&$\nu=.49999$  \\
    \hline
     \multicolumn{6}{| c |}{option I}\\
    \hline
      $2^{-3}$&$0.8992$&$0.9361$&$0.9405$&$0.9409$&$0.9409$\\
      $2^{-4}$&$0.9196$&$0.9580$&$0.9625$&$0.9630$&$0.9630$\\
      $2^{-5}$&$0.9251$&$0.9639$&$0.9684$&$0.9689$&$0.9690$\\
      $2^{-6}$&$0.9267$&$0.9656$&$0.9701$&$0.9706$&$0.9706$\\
              \hline
    \multicolumn{6}{| c |}{option II}\\
            \hline
      $2^{-3}$&$1.3311$&$1.3561$&$1.3591$&$1.3594$&$1.3594$\\
      $2^{-4}$&$1.3435$&$1.3701$&$1.3732$&$1.3735$&$1.3736$\\
      $2^{-5}$&$1.3468$&$1.3737$&$1.3769$&$1.3773$&$1.3773$\\
      $2^{-6}$&$1.3477$&$1.3748$&$1.3780$&$1.3783$&$1.3783$\\
                    \hline
  \end{tabular}   
  \end{center}
\end{table}
\begin{table}[ht!]
 \caption[]{{Test problem 1: Effectivity indices for fixed finite element mesh size $h=2^{-6}$.} }
    \label{small_sigma_p31}
\begin{center}
    \begin{tabular}{ | l | l | l | l | l | l | l | l | l | l | l | }
      \hline   
    \multicolumn{1}{| c |}{$k$}&{$\nu=.4$ }& $\nu=.49$&$\nu=.499$&$\nu=.4999$&$\nu=.49999$  \\
    \hline
     \multicolumn{6}{| c |}{option I}\\
    \hline
    $2$&$0.9923$&$1.0287$&$1.0330$&$1.0334$&$1.0335$\\
      $3$&$0.9266$&$0.9656$&$0.9701$&$0.9706$&$0.9706$\\
      $4$&$0.9265$&$0.9654$&$0.9700$&$0.9704$&$0.9705$\\
      $5$&$0.9265$&$0.9654$&$0.9700$&$0.9704$&$0.9705$\\
              \hline
    \multicolumn{6}{| c |}{option II}\\
            \hline
       $2$&$1.3937$&$1.4198$&$1.4229$&$1.4233$&$1.4233$\\
      $3$&$1.3477$&$1.3748$&$1.3780$&$1.3783$&$1.3783$\\
      $4$&$1.3476$&$1.3747$&$1.3779$&$1.3782$&$1.3782$\\
      $5$&$1.3476$&$1.3747$&$1.3779$&$1.3782$&$1.3782$\\
                    \hline
  \end{tabular}   
\end{center}
\end{table}

\subsection{\rbl{Singular problem,} mixed boundary conditions}
To test our error estimation strategy in a more realistic setting we next consider  a  problem with  a mixed boundary condition (so that $\partial D_N \neq \emptyset$).
Specifically, we take the unit square domain with a homogeneous Neumann boundary condition 
on the right edge $\partial D_N=\{1\}\times(0,1)$ and a zero Dirichlet boundary condition for the
 displacement on $\partial D_D=\partial D\setminus\partial D_N$.    The uncertain Young's modulus has mean value one and is given by the representation}\footnote{This parametric representation is commonly used in the literature (see, for example,~\cite{EIGEL2014247}). It characterises one of 
  several test problems that are  built into the S-IFISS toolbox~\cite{SIFISS}.}
\begin{align}\label{uncertainE}
E(\bm{x},\bm{y})= 1 + \sum_{m=1}^\infty \alpha_m \cos(2\pi \beta_1(m) x_1) \cos(2\pi \beta_2(m)x_2)y_m , \quad \bm{x} \in D, \bm{y} \in \Gamma,
\end{align}
 where $\Gamma=\Pi_{m=1}^{\infty} \Gamma_{m}$ and $y_m \in \Gamma_{m}:=[-1,1]$. For each $m\in\mathbb{N}$, 
 \begin{align}
 \beta_1(m)=m-k(m)(k(m)+1)/2\quad \mbox{and}\quad \beta_2(m)=k(m)-\beta_1(m)
 \end{align}
 where $k(m)=-1/2+\sqrt{1/4+2m}$ and $\alpha_m=\bar{\alpha}m^{-\tilde{\sigma}}$ for fixed $\tilde{\sigma} >1$ and $0<\bar{\alpha}< 1/\zeta(\tilde{\sigma})$,  where $\zeta$ is the Riemann zeta function.
The sparse  (Galerkin) high-dimensional system of linear equations 
that is generated at each step of the adaptive algorithm   is solved 
using a bespoke  MINRES  solver (EST\_MINRES) in combination with the efficient
preconditioning strategy presented in~\cite{KPSpre}.

We present results for the case of  a horizontal body force  $\bm{f}=(0.1,0)^{\top}$. This
generates an  exact displacement solution that is symmetric about the line $y=1/2$. 
The problem has limited regularity in the compressible case: for $\nu=0.4$ 
there are strong singularities at the two corners 
where the boundary condition changes from essential to natural.
The singularities become progressively weaker in the 
incompressible limit and their effect on the solution is  imperceptible
 when $\nu = 0.49999$.

To ensure a reasonable level of accuracy in the singular cases we used $\bm{P}_2$--$P_{1}$--$P_{1}$ triangular approximation\footnote{The
(Taylor--Hood) combination $\bm{P}_2$--$P_{1}$ is the best known inf--sup stable approximation pair;
see~\cite[Sect.\,3.3.3]{HDA}.}
in combination with spatial adaptivity. The  number of  displacement degrees of freedom in the initial mesh ${\cal T}_{h_0}$
was $162$. The adaptive algorithm was terminated when the total number of degrees  of freedom (spatial $\times$ parametric)
exceeded $5\cdot 10^5$  when $\nu=0.4$ and $10^5$ when $\nu=0.49999$.
 We checked the convergence of Algorithm~\ref{algor_1} for  two choices of coefficient $\alpha_m$ in (\ref{uncertainE}):  $\tilde{\sigma}=2$ (slow decay)  
 and $\tilde{\sigma}=4$ (fast decay). 
Results for the slow decay case are shown in Figs.~\ref{ex2uqv01v02errn11}--\ref{ex2uqv01v02meshi11}.
We make the following observations.
 \begin{itemize}
 \item
The rate of convergence is $O(n^{-1/2})$ (where $n$ is the total number of degrees of freedom)
and is  independent of the Poisson ratio.

 \item At all refinement steps, the error estimate $\eta$ is dominated by the spatial error contribution  $\eta_{h^*,\Lambda}$
 in the compressible case, that is when $\nu=0.4$. In contrast,  the parametric error contribution dominates 
at  several steps of the algorithm  in the nearly incompressible case.
 \item
Looking at Fig.~\ref{ex2uqv01v02acraci11} we see that  twice as many parameters (and indices) are activated in the nearly incompressible case. The number of adaptive steps would be  reduced if we were to probe more than one additional parameter when constructing the detail index set \eqref{indsetQ}.  The computational experiments reported in \cite{cpb18} show that a more efficient algorithm may be obtained in the slow decay case if multiple parameters are activated at every step.
\item The number of displacement degrees of freedom in the mesh when
the algorithm terminated  was  $26,094$  when $\nu=0.4$ and  $2,070$
in the nearly incompressible case. These meshes are
shown in Fig.~\ref{ex2uqv01v02meshi11} and clearly illustrate the
influence of the spatial singularities in the compressible case.
   \end{itemize}

 \begin{figure}[th!]
 \begin{center}
\includegraphics[width=.43\textwidth]{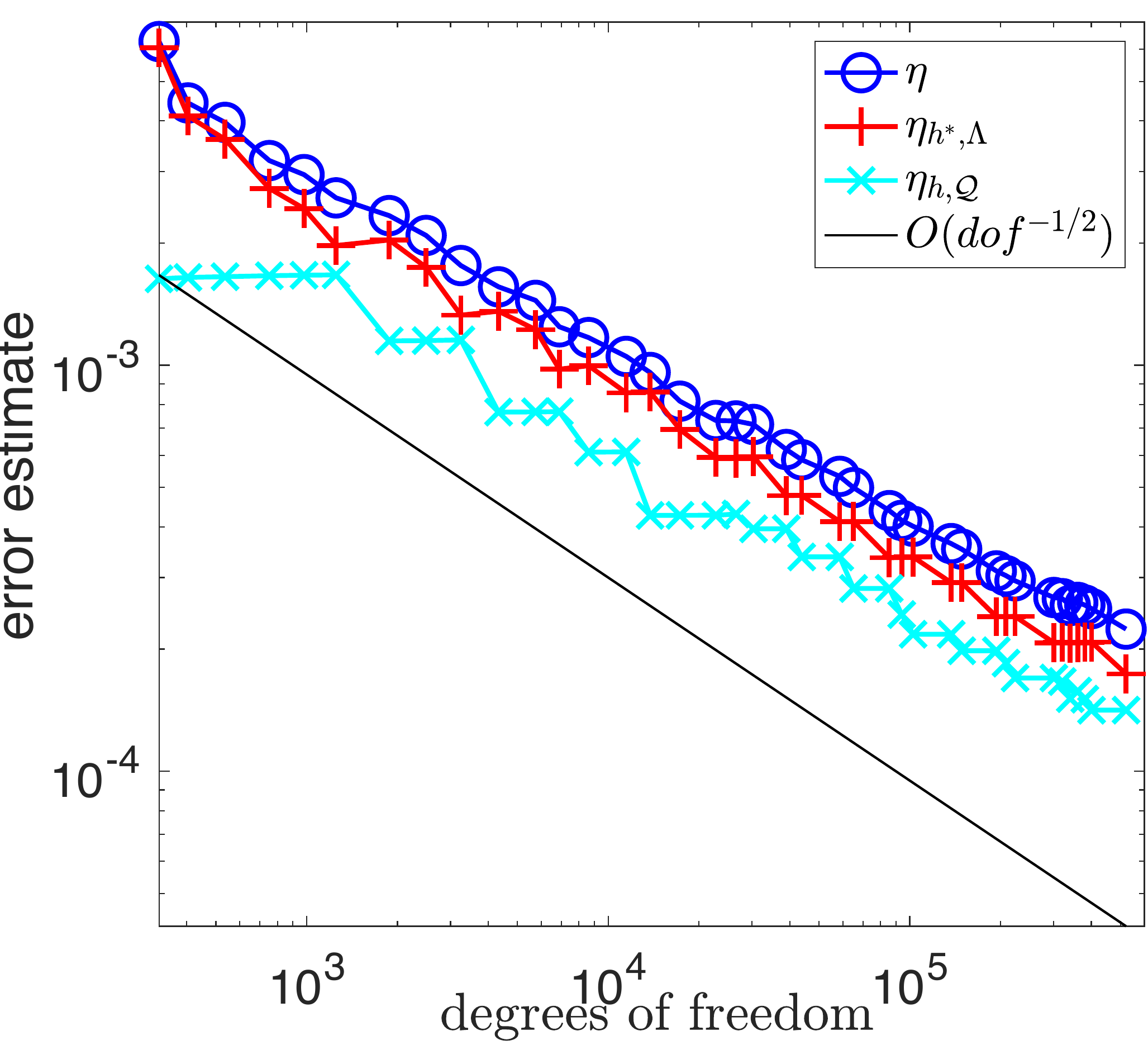} 
\includegraphics[width=.43\textwidth]{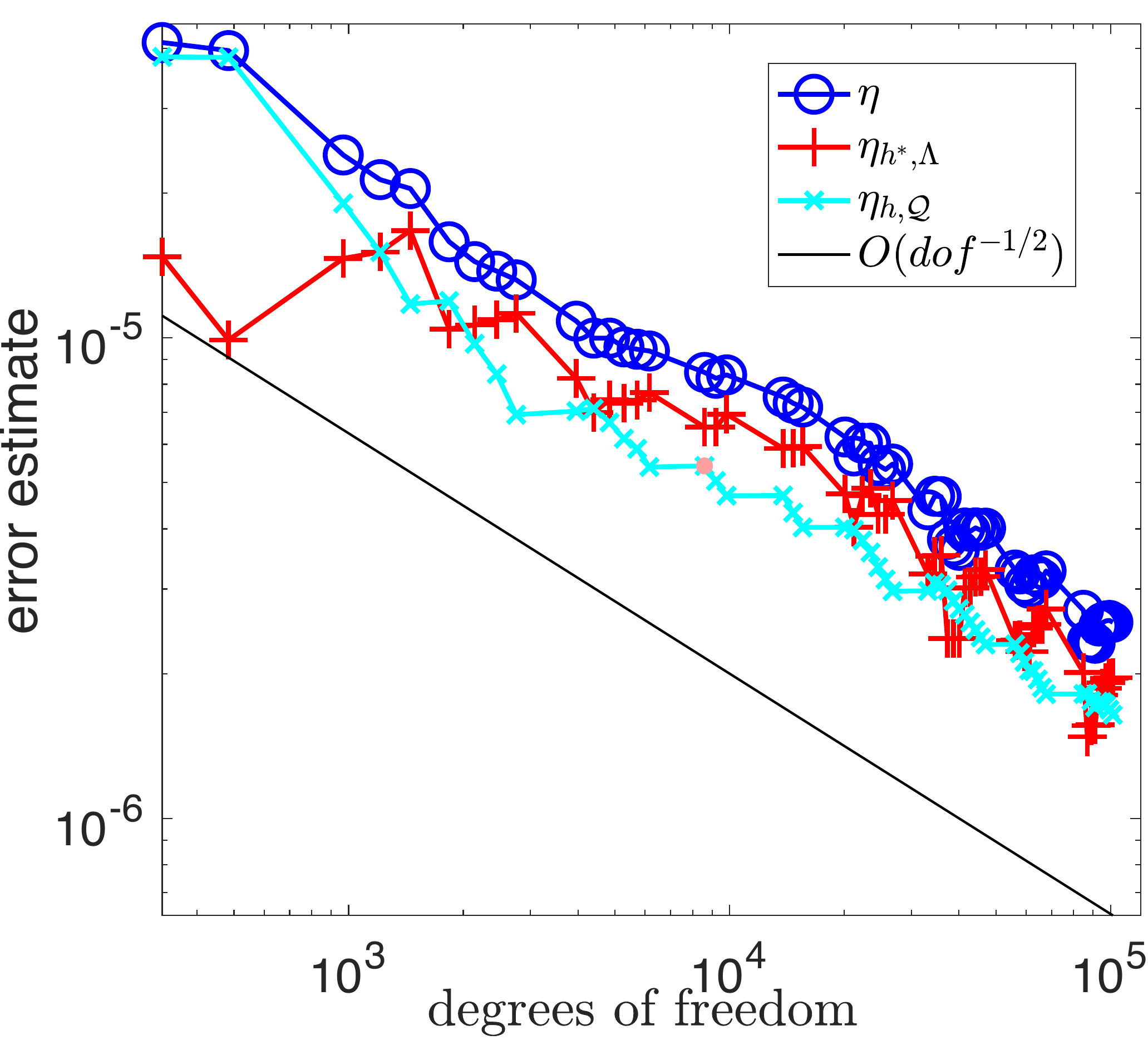}
\caption{{Test problem 2 (slow decay, $\tilde{\sigma}=2$): Estimated error at each step 
of Algorithm~\ref{algor_1} for $\nu=.4$ (left);   $\nu=0.49999$  (right). } }
\label{ex2uqv01v02errn11}
\end{center}
\medskip

\begin{center} 
\includegraphics[width=.47\textwidth]{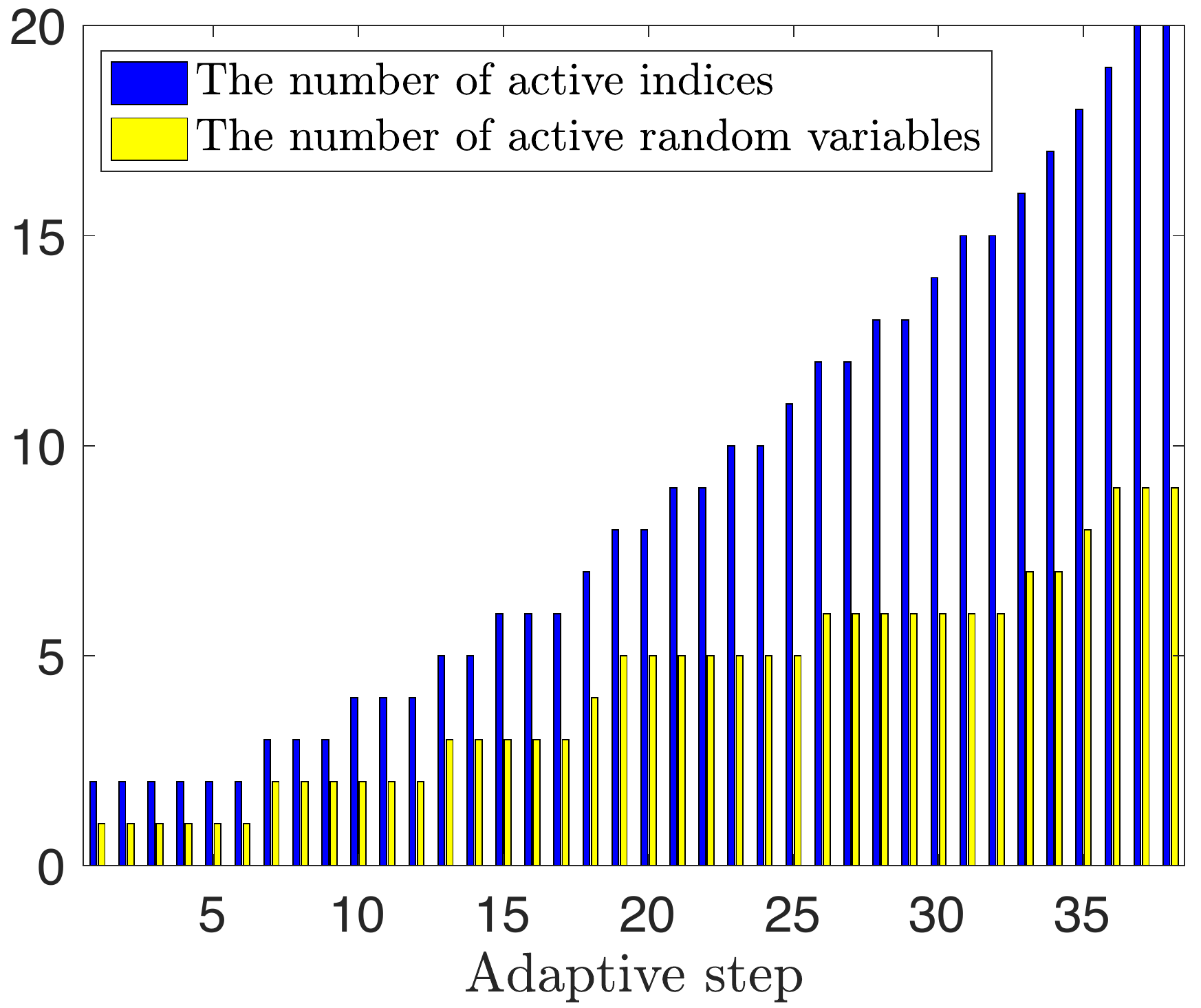}
\includegraphics[width=.47\textwidth]{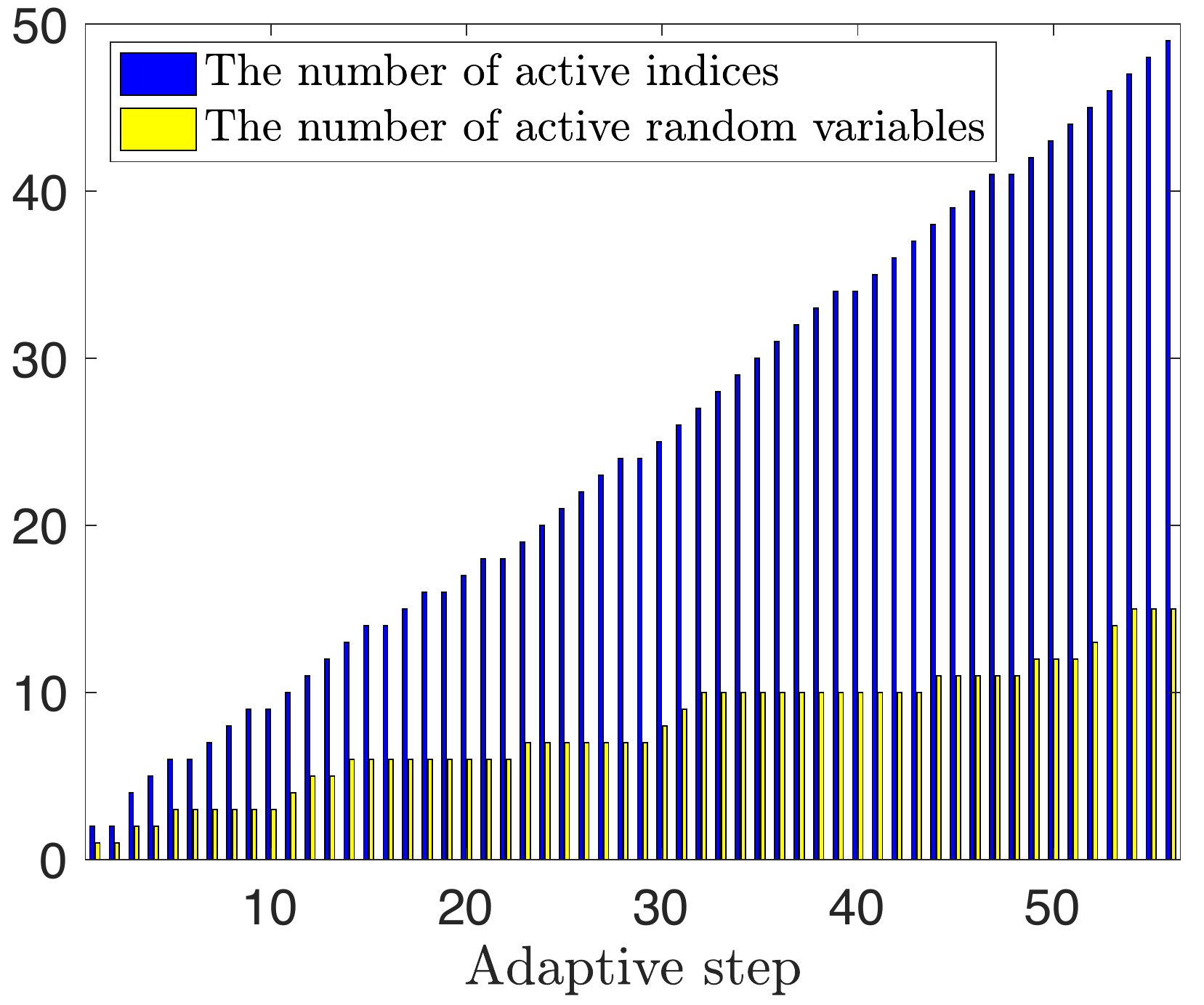}
\caption{{Test problem 2 (slow decay, $\tilde{\sigma}=2$): The number of active multi-indices $\balpha$ 
and active random variables $y_{m}$ at each step of the Algorithm~\ref{algor_1} for  $\nu=.4$ (left);   $\nu=0.49999$  (right). } }
\label{ex2uqv01v02acraci11}
\end{center}
\medskip

\begin{center} 
\includegraphics[width=.42\textwidth]{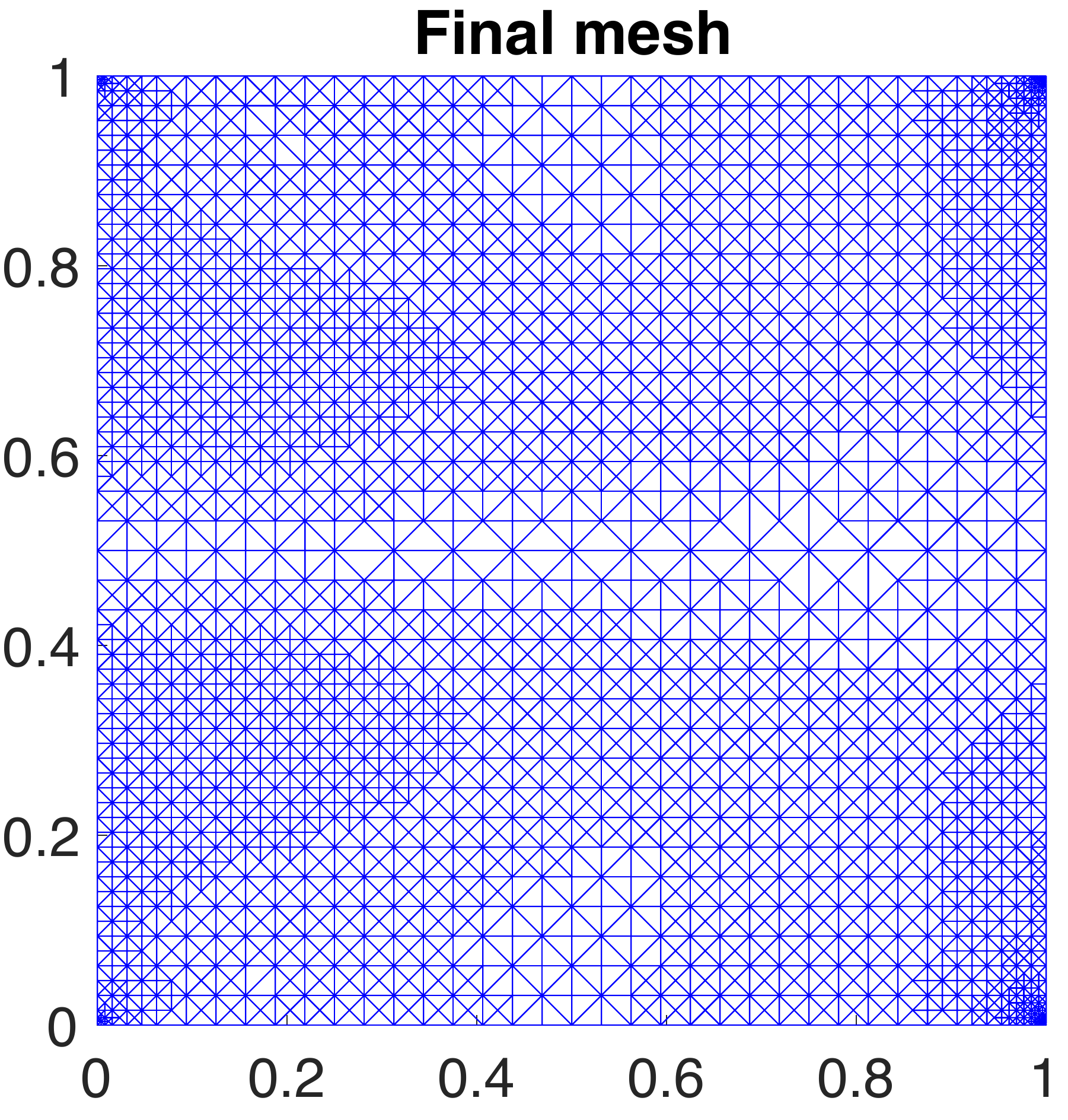}
\includegraphics[width=.42\textwidth]{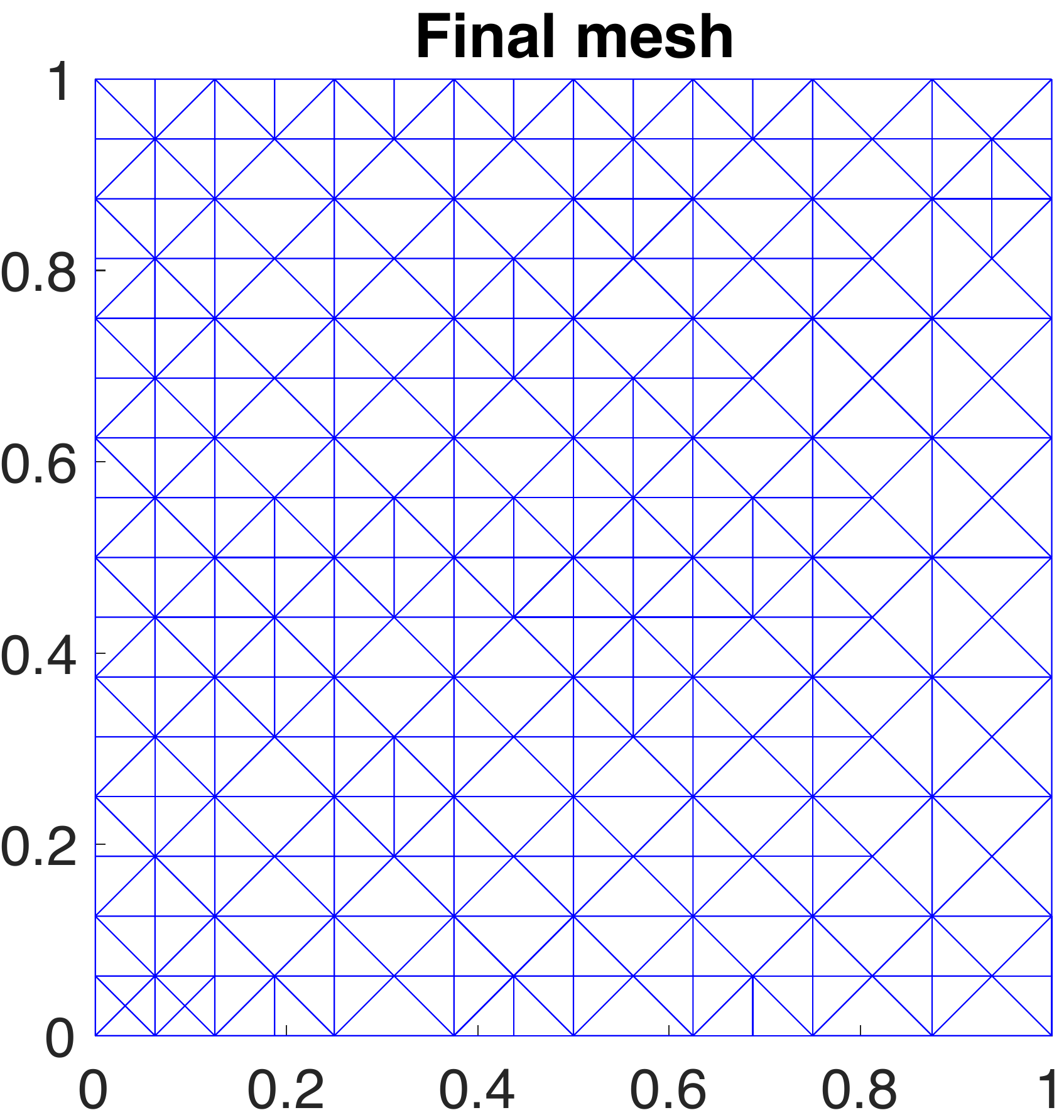}
\caption{{Test problem 2 (slow decay, $\tilde{\sigma}=2$):  triangular mesh at the step when
the  target number of degrees of freedom was reached for $\nu=.4$ (left);   $\nu=0.49999$  (right). } }
\label{ex2uqv01v02meshi11}
\end{center}
\end{figure}


\rbl{Analogous results obtained in the fast decay case are shown in  Figs.~\ref{ex2uqv01v02errn12}--\ref{ex2uqv01v02acraci12}.}  We make two final observations.
 \begin{itemize}
 \item Once again, the rate of convergence is $O(n^{-1/2})$ (where $n$ is the total number of degrees of freedom)
and is  independent of the Poisson ratio. 
  \item Comparing  Fig.~\ref{ex2uqv01v02acraci11} with Fig.~\ref{ex2uqv01v02acraci12} we see that the number of parameters (and indices) that are activated in the fast decay case is 
much less than the number that were activated in the slow decay case for the same total number of degrees of freedom. 
\end{itemize}

\begin{figure}[th!]
\begin{center} 
\includegraphics[width=.42\textwidth]{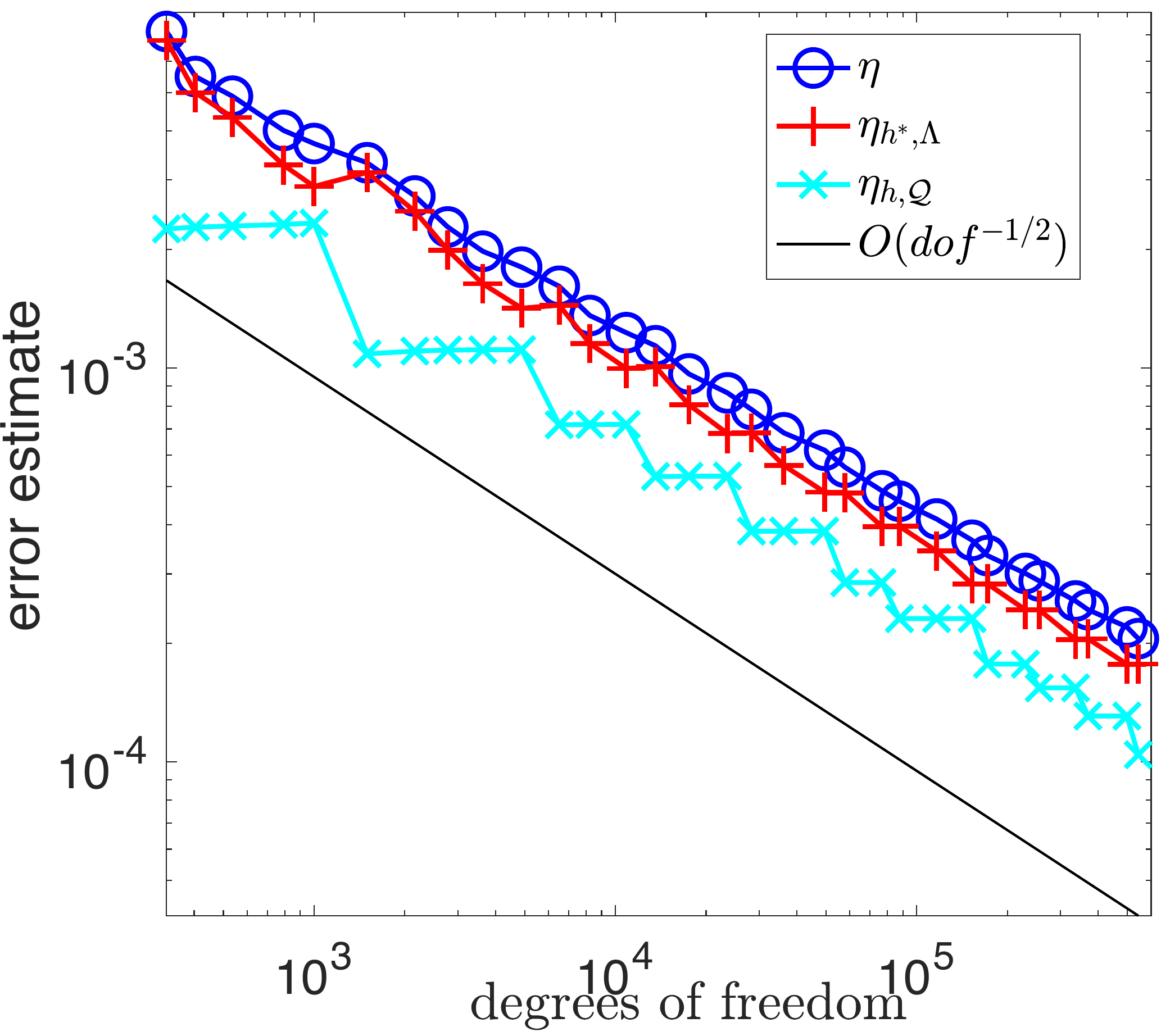}
\includegraphics[width=.42\textwidth]{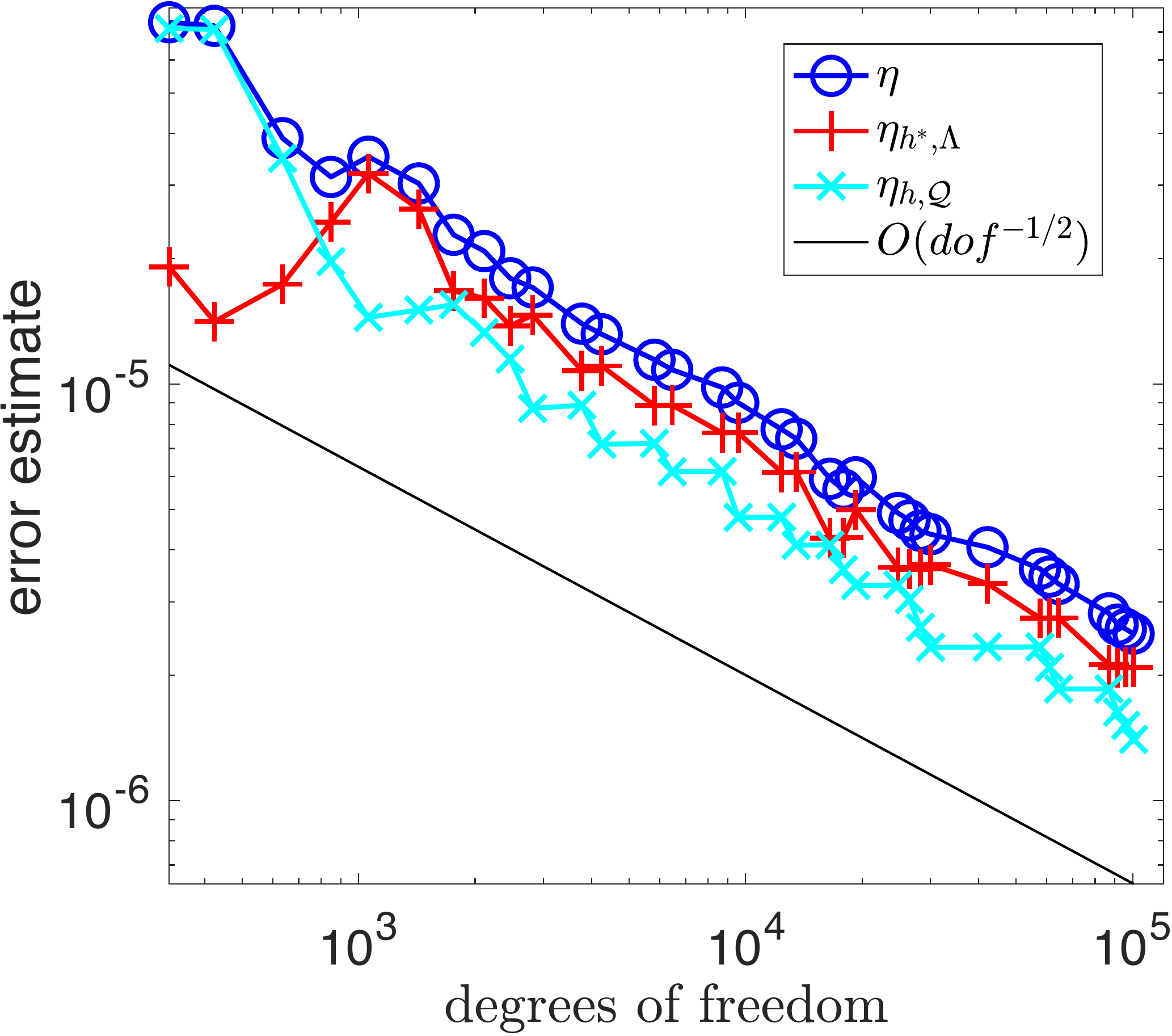}
\caption{{Test problem 2 (fast decay, $\tilde{\sigma}=4$): 
Estimated error at each step of Algorithm~\ref{algor_1} for $\nu=.4$ (left);   $\nu=0.49999$  (right). } }
\label{ex2uqv01v02errn12}
\end{center}
\smallskip

\begin{center} 
\includegraphics[width=.45\textwidth]{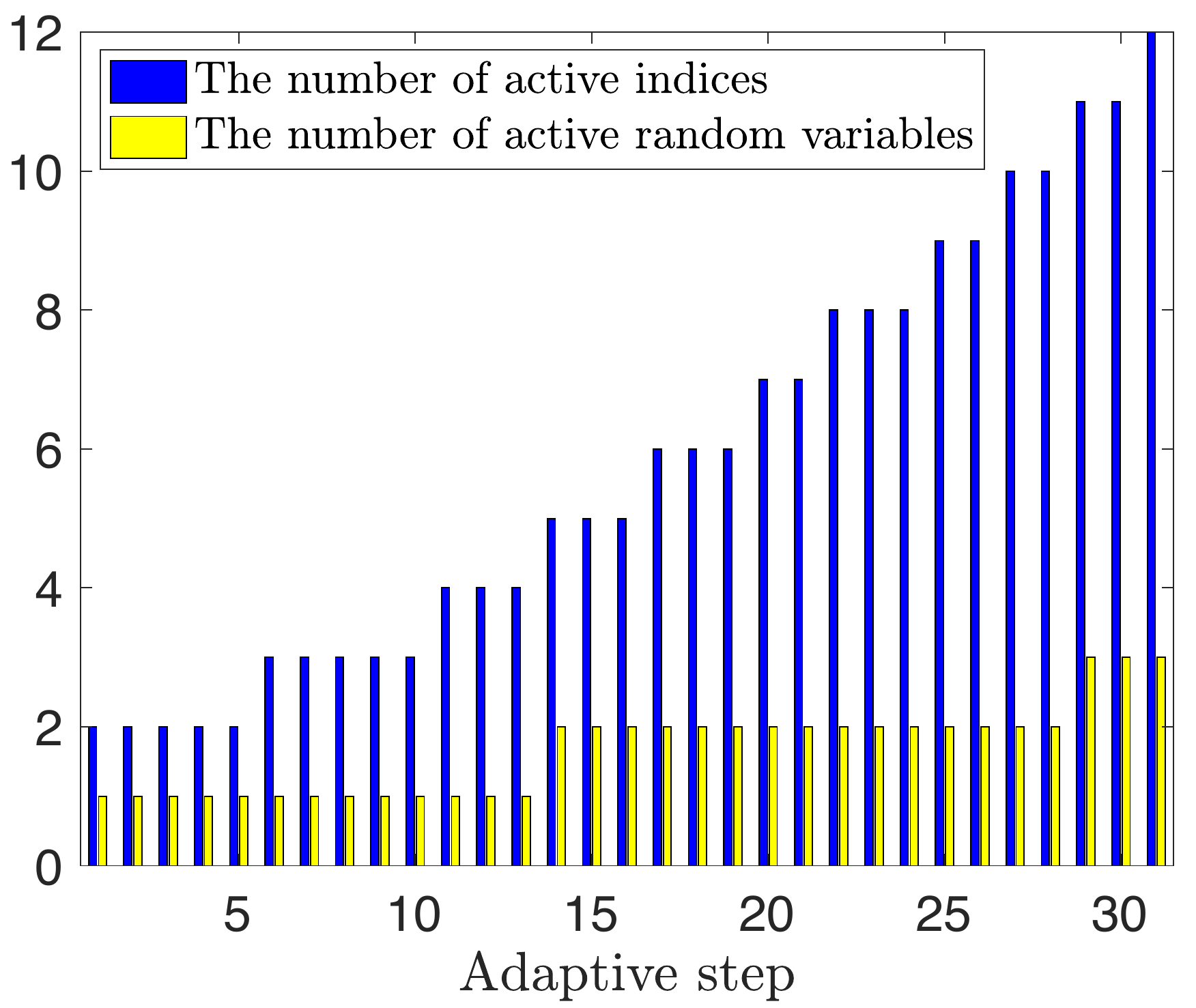}
\includegraphics[width=.45\textwidth]{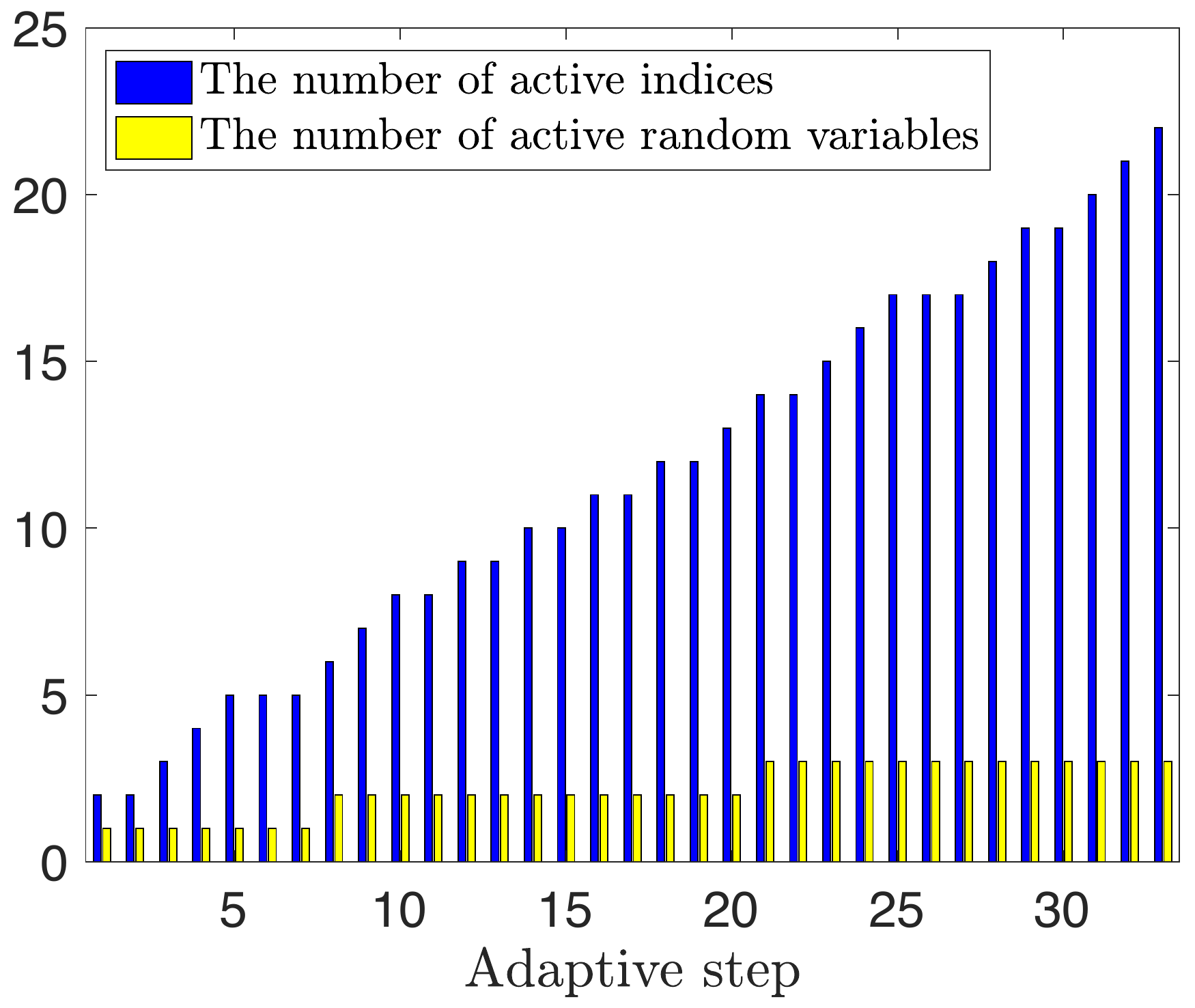}
\caption{{Test problem 2 (fast decay, $\tilde{\sigma}=4$): 
The number of active multi-indices $\balpha$ and active random variables $y_{m}$ 
at each step of the Algorithm~\ref{algor_1} for  $\nu=.4$ (left);   $\nu=0.49999$  (right). } }
\label{ex2uqv01v02acraci12}
\end{center}
\end{figure}

\section{Summary} 
Our thesis is that efficient adaptive algorithms hold the key to effective computational solution of PDEs of elliptic type with uncertain
material coefficients. This paper has two  important contributions,  building on earlier work  for scalar diffusion problems.
First, we have shown that  mixed formulations of 
elasticity equations with parametric uncertainty can be solved in a black-box fashion. 
We believe that this opens the door to practical engineering analysis of
structures with uncertain material coefficients.
Second, in contrast to other work in this area, which typically estimates a posteriori errors 
by taking norms of residuals, our approach can give accurate proxies of 
potential error reductions that would occur if different refinement strategies were pursued.   
Extensive numerical testing  confirms that effectivity indices close to unity can be maintained throughout the refinement process.

\bibliographystyle{siam}
\bibliography{kps_upd}
\end{document}